\documentclass[a4paper,12pt,reqno]{amsart}
\usepackage{geometry}
\usepackage[utf8]{inputenc}

\usepackage{amsfonts}  % Additional fonts
\usepackage{amsthm}    % Theorem environments
\usepackage{amsmath}   % Enhanced math features
\usepackage{bm}        % Bold math symbols
\usepackage{bbm}
\usepackage{amssymb}   % Additional math symbols
\usepackage{commath}	%Norm symbols
\usepackage{mathrsfs}
\usepackage{listings}
\usepackage{xcolor}

\usepackage[dvipsnames, table]{xcolor}	%Table and text colour
\usepackage{graphicx}  % Enhanced graphics inclusion
\usepackage{tikz}      % Vector graphics
\usepackage{tikz-cd}   % Commutative diagrams
\usepackage{pgfplots}  % Plotting graphs
\pgfplotsset{compat=1.18}
\usetikzlibrary{arrows.meta, positioning, calc, math}  % TikZ libraries
\tikzset{>={Classical TikZ Rightarrow[width=7pt, length=5pt]}}
\usetikzlibrary{cd}  % More TikZ libraries

\usepackage{array}     % Extended table features
\usepackage{xtab}      % Break long tables
\usepackage{url}
\usepackage{tabularx}
\usepackage{colortbl}
\usepackage{xltabular}

\usepackage{comment}
\usepackage[T1]{fontenc} %accent
\usepackage{lmodern} %retain original font
\usepackage{multicol}
\usepackage{pdflscape, adjustbox}	%landscape one page/section
\usepackage{parskip}       % Paragraph spacing
\usepackage{mathrsfs}      % Script fonts
\usepackage{fancyhdr}      % Fancy headers and footers
\usepackage[hidelinks]{hyperref}  % Hyperlinks without color
\usepackage{mathtools}     % More math symbols
\usepackage[shortlabels]{enumitem}  % Enhanced lists
\usepackage{witharrows}    % Arrows in math environments
\usepackage{float}         % Improved float control
\usepackage{makecell}
\usepackage{cleveref}	%auto type `Theorem/Proposition/Remark...' when using \ref
\usepackage{comment}

\newcommand{\Z}{\ensuremath{\mathbb{Z}}}% integers
\newcommand{\Q}{\ensuremath{\mathbb{Q}}}% rationals
\newcommand{\N}{\ensuremath{\mathbb{N}}}% natural numbers
\newcommand{\R}{\ensuremath{\mathbb{R}}}% real numbers
\newcommand{\C}{\ensuremath{\mathbb{C}}}% complex numbers
\newcommand{\F}{\ensuremath{\mathbb{F}}}% finite field
\newcommand{\q}{\ensuremath{\mathfrak{q}}}% small prime ideal q
\newcommand{\p}{\ensuremath{\mathfrak{p}}}% small prime ideal p
\newcommand{\m}{\ensuremath{\mathfrak{m}}}% small prime ideal m
\newcommand{\A}{\ensuremath{\mathfrak{a}}}% small prime ideal a
\newcommand{\f}{\ensuremath{\mathfrak{f}}}% small prime ideal f
\newcommand{\g}{\ensuremath{\mathfrak{g}}}% small prime ideal d
\newcommand{\PP}{\ensuremath{\mathcal{P}}}% power set font p 
\newcommand{\OO}{\ensuremath{\mathcal{O}}}% ring of integer O

\DeclareMathOperator{\Ann}{ann}
\DeclareMathOperator{\Ord}{ord}

\DeclareMathOperator{\Nm}{Nm}

\DeclareMathOperator{\End}{End}
\DeclareMathOperator{\Gcd}{gcd}
\DeclareMathOperator{\Sup}{sup}
\DeclareMathOperator{\Lcm}{lcm}
\DeclareMathOperator{\Norm}{Nm}
\DeclareMathOperator{\Ker}{ker}
\DeclareMathOperator{\Frob}{Frob}
\DeclareMathOperator{\Deg}{deg}
\DeclareMathOperator{\Art}{Art}
\DeclareMathOperator{\Gal}{Gal}
\DeclareMathOperator{\Cl}{Cl}

\newtheorem{prop}{Proposition}[section]

\newtheorem{theorem}[prop]{Theorem}
\newtheorem{lemma}[prop]{Lemma}

\newtheorem{cor}[prop]{Corollary}

\theoremstyle{definition}
\newtheorem{remark}[prop]{Remark}
\newtheorem{example}[prop]{Example}

\newtheorem{defn}[prop]{Definition}

\numberwithin{equation}{section}% standard style numbering, nothing special here

\crefname{prop}{Proposition}{Propositions}
\Crefname{prop}{Proposition}{Propositions}

\crefname{cor}{Corollary}{Corollaries}
\Crefname{cor}{Corollary}{Corollaries}

\crefname{defn}{Definition}{Definitions}
\Crefname{defn}{Definition}{Definitions}
\include{contents}
\begin{document}
\author{Edison H L Au-Yeung}
\title[Higher rank elliptic sieve via CM endomorphisms]{Higher rank elliptic sieve via CM endomorphisms}
\address{Mathematics Institute, Zeeman Building, University of Warwick, Coventry CV4 7AL, England}
\email{hang.au-yeung@warwick.ac.uk}
\date{\today}
\thanks{The author is supported by the Warwick Mathematics Institute Centre for Doctoral Training, and gratefully acknowledges funding from the University of Warwick.}

\begin{abstract}
A theorem of Landau and Ramanujan states that almost all integers are not sums of two squares (when ordered by their absolute value). Equivalently, almost all integers are not norms from the Gaussian rationals $\Q(i)$. In this paper, we prove a version of this norm result for elliptic curves $E$. Let $E$ be an elliptic curve over a number field $K$ with complex multiplication, such that $K$ contains the endomorphism ring of $E$.  For a fixed non-torsion point $P \in E(K)$ and an abelian extension $L/K$, we show that subject to a technical condition, the $y$-coordinate of almost all multiples of $P$ does not belong to the norm group of $L$. The proof is an adaptation of an argument of Bhakta et al.\ , whose main tools are elliptic divisibility sequences and a sieve technique known as the elliptic sieve. To account for the points in $E(K)$ arising from the additional endomorphisms, we apply the elliptic sieve to elliptic nets, a higher-rank generalisation of elliptic divisibility sequences. This approach demonstrates the possibility of applying the elliptic sieve to $\Z$--linear combination of points on elliptic curves and the usefulness of elliptic nets to counting problems on elliptic curves.
\end{abstract}

\maketitle
\setcounter{tocdepth}{1}
\tableofcontents
\section{Introduction}
\subsection{Motivation}\label{motivation sieve}
There is a well-known theorem, due to Landau \cite{Landau} and Ramanujan, claiming that almost all integers are not sums of two squares, when ordered by absolute value. In \cite{Brauer}, Bhakta et al.\ proved a version of this result for elliptic curves.

\begin{theorem}\cite[Theorem 1.1]{Brauer}
Let $E$ be an elliptic curve over $\Q$ given by an integral Weierstrass equation. Let $P\in E(\Q)$ have infinite order with $P \in E(\R)^{0}$, where $E(\R)^{0}$ denotes the connected component of the identity of $E(\R)$. Then there exists $\omega=\omega(E, P)>0$ such that 
\[
\#\{n \in \Z\colon |n|\leq B, y(nP) \text{ is a sum of two squares}\} \ll_{E, P} \frac{B}{\left(\log B\right)^{\omega}}. 
\]
\end{theorem}

Note that the problem of writing an integer as a sum of two squares is equivalent to the problem of an integer being a norm from the Gaussian integers $\Z[i]$. Therefore, as an appplication, Bhakta et al.\ also generalised the above theorem to the norm problem in \cite{Brauer}, restricted to certain field extensions. 

\begin{theorem}\cite[Theorem 1.7]{Brauer} 
Let $K/\Q$ be a number field which contains a cyclic subfield which is not totally real. Let $E$ be an elliptic curve over $\Q$ given by an integral Weierstrass equation and let $P \in E(\Q)$ be a point of infinite order with $P \in E(\R)^{0}$. Then there exists $\omega=\omega(E, K)>0$ such that 
\[
\#\{n \in \Z\colon |n|\leq B, y(nP) \text{ is a norm from } K\} \ll_{E, P, K} \frac{B}{\left(\log B\right)^{\omega}}. 
\]
\end{theorem}

One of the main tools used by Bhakta et al.\ is a sieving technique named \textit{the elliptic sieve}, a term previously introduced by Kowalski in \cite[Section 4.4]{Kowalski}. Despite sharing the same name, the two sieves are fundamentally different. Kowalski's elliptic sieve is an application of the classical large sieve, which sieves the Mordell--Weil group $E(\Q)$ with respect to the reduction homomorphism $E(\Q) \rightarrow E(\F_p)$. The goal of Kowalski's sieve is to show that most rational points have denominators divisible by many primes (\cite[Theorem 1.1]{Kowalski}); thus, information modulo $p$ is sufficient. In contrast, for the norm problem, one must know the $p$-adic valuation of a rational point to determine whether its coordinate(s) could be norm(s) from a given field extension. Consequently, Bhakta et al.\ actually sieved modulo $p^{n_p}$ for varying primes $p$ and growing exponents $n_p$.

One possible generalisation of the elliptic sieve method is to extend the norm problem to $\Z$-linear combinations of points: instead of examining $y(nP)$, we consider $y(nP+mQ)$ for non-torsion points $P, Q \in E(\Q)$. However, the authors of \cite{Brauer} note that their elliptic sieve method cannot currently handle cases where the rank is greater than $1$. In the rank $1$ setting, applying the sieve involves examining multiples of a fixed point $[\ell]P$, where $\ell$ runs through all rational primes. One can show that there is a positive proportion of primes $\ell$ that both occur as the order of $P \bmod p$ in the finite group $E(\F_p)$ and simultaneously violate the norm conditions (\cite[Lemma 4.2]{Brauer}). These are the primes that Bhakta et al.\ sieved. For elliptic curves of rank higher than $1$, simply analysing $\Z$-linear combinations $[n]P+[m]Q$ for primitive lattice points $(n, m) \in \Z^2$ yields an insufficient proportion of such points to make meaningful conclusions from the sieve. Furthermore, there is also a lack of effective methods for obtaining $p$-adic information regarding the sum $[n]P+[m]Q$. For example, we can use the subgroup filtration property (see (\ref{filtration})) for elliptic curves of rank $1$, which is not applicable for $\Z$-linear combinations $[n]P+[m]Q$. We refer the interested reader to \cite[Sections 1.3, 1.4, and Lemma 4.3]{Brauer} for further details.

In this regard, the main objective of this paper is to demonstrate the possibility of applying the elliptic sieve to $\Z$-linear combinations $[n]P+[m]Q$ for elliptic curves with complex multiplication, which allows us to prove analogues of the above results. If an elliptic curve has complex multiplication by $\Z[\tau]$, an order in an imaginary quadratic field, then we are allowed to consider the point $[a+b\tau]P$ as a $\Z$--linear combination of two points $[a]P+[b]([\tau]P)$, where $a, b \in \Z$.  This effectively enables the application of the elliptic sieve in a rank $2$ scenario. Although this requires the elliptic curve to be defined over a number field containing its endomorphism ring (and thus no longer over $\Q$), this approach allows us to generalise the result to arbitrary abelian extensions.

\begin{theorem}\label{Main ideal norm result}
Let $E/K$ be an elliptic curve over a number field $K$ given by an integral Weierstrass equation, and suppose that $E$ has complex multiplication by an order $\OO$ in the imaginary quadratic field $F\subseteq K$. Let $L/K$ be a finite abelian extension with conductor $\f=\mathfrak{f}_{L/K}$.

Let $P\in E(K)$ be a non-torsion point and $D_{\alpha}$ be the elliptic denominator sequence associated to $P$, to be defined in \Cref{denominator sequence def}. If there exists $\alpha_0\in\mathcal{O}$ with $\alpha_0 \in (\OO/\pi\OO)^{\times}$, where $\pi \in \OO$ is defined in \Cref{Lambda estimate}, such that $(D_{\alpha_0}, \f)=1$ and the image of $D_{\alpha_0}$ under the Artin map $\Art_{L/K}^{\f} \colon I_{K}^{\f} \longrightarrow \Gal(L/K)$ is nontrivial, then there exists $\omega=\omega(E,P)>0$ such that, for every $\varepsilon>0$,
\[
\#\left\{
\alpha\in\mathcal{O}\colon |\alpha|\leq B,\;
y(\alpha P)\text{ is an ideal norm from }L
\right\}
\ll_{E,P,L}
\frac{B}{(\log B)^{\omega-\varepsilon}},
\]
where the implicit constant has dependence on $P$ of the form $\exp\left(
O_{E,L}\left(\widehat h(P)/\log\widehat h(P)\right)\right)$.
\end{theorem}

Similar to the original work, the main tool that allows us to analyse $p$--adic information of $\Z$--linear combination of points is elliptic nets, which is a higher rank generalisation of elliptic divisibility sequences (EDSs). Elliptic nets are first defined by Stange in \cite{Stange1}, which was then used by the author to define CM EDS in \cite{Edison1}. They are defined via the denominators of the $x$--coordinates of the multiples of a fixed point $P$. Under our assumptions, \(E\) has complex multiplication by an order \(\mathcal O\) in an imaginary quadratic field \(F\), with \(F\subseteq K\). Therefore, \(K\) is totally imaginary. In comparison with the original paper \cite{Brauer}, these assumptions allow us to avoid issues with signs of elliptic nets, hence simplifying our work and making the arguments easier to follow. On the other hand, we also managed to show periodicity of CM EDS, which is a key property for applying the sieve method as in \cite{Brauer}. We recall the relevant definitions in \Cref{EDS for CM curves}.

We end this section by remarking that the technical condition in the theorem concerning the existence of the denominator ideal $D_{\alpha}$ is necessary. In \Cref{sieve example}, we provide three examples to illustrate this theorem. In \Cref{example Q(i)} and \Cref{sieve example 2}, we consider the elliptic curve $E \colon y^2 = x^3 - 2x$, which has complex multiplication by $\mathbb{Z}[i]$, defined over number fields with different class numbers. In both examples, we consider the $EDS$ associated to the point $P=(-1, 1)$ and show that it satisfies all the hypotheses, therefore the conclusion of the theorem holds. In contrast, \Cref{sieve example 3} considers the elliptic curve 
\[
E/\Q(\sqrt{-2})\colon y^2=x^3+x^2-3x+1
\]
and the EDS associated to the point $P=(0, -1)$. In this example, there does not exist a multiple $\alpha_0 \in \OO$ such that the hypotheses are all satisfied and the conclusion fails. More specifically, our method is only able to show that roughly $50\%$ of the points in $E(\Q(\sqrt{-2}))$ are candidates for being ideal norms from the $8$--th cyclotomic field $\Q(\zeta_8)$. 

\begin{remark}
Similarly, the result by Bhakta et al.\ \cite[Theorem 1.7]{Brauer} includes a technical condition, namely $P \in E(\mathbb{R})^{0}$, which is necessary for a similar reason: the image of the denominator ideal under the Dirichlet character (Kronecker symbol) is nontrivial. For comparison and a more detailed explanation, we refer the reader to the proofs of \cite[Theorem 1.6, Proposition 3.9]{Brauer} and compare it with \Cref{Lambda estimate}.
\end{remark}

\subsection{Outline of the paper}
The paper is organized as follows. In \Cref{preliminary}, we recall the necessary facts about elliptic curves with complex multiplication over $K_{\p}$ and explain the use of (ray class) group characters to detect elements in the norm group, generalising the use of Dirichlet characters in \cite{Brauer}. In \Cref{CM EDS periodicity section}, we provide a brief exposition of CM EDS and then prove their periodicity. After proving the main theorem in \Cref{Proof of sieve main result}, we give examples illustrating our main result in \Cref{sieve example} and the conclusion could fail if the technical conditions are not satisfied.

\subsection{Notations and conventions}\label{Notations}
Throughout the paper, unless otherwise specified, we will use the following notations and conventions. By Weierstrass equation with coefficients $a_i$, we mean an elliptic curve $E/K$ defined by the equation below:
\begin{equation}\label{EC}
E\colon y^2 + a_1xy + a_3y = x^3 + a_2x^2 + a_4x + a_6; a_i \in \OO_{K}.
\end{equation}

We also assume that $E$ has complex multiplication by an order $\OO$ in an imaginary quadratic field $F\subseteq K$, generated by $\{1, \tau\}$, so $\OO=\Z[\tau]=\End(E)$. Under this assumption, $K$ must be a totally imaginary number field; we also restrict our attention to only finite abelian extensions $L/K$. By $O_{a, \ldots, z}(A)$ we mean a quantity with absolute value at most $CA$ for some positive constant $C$ depending on $a, \ldots, z$ only; if the subscripts are omitted the implied constant is absolute. We write $A\ll_{a, \ldots, z}B$ for $A=O_{a, \ldots, z}(B)$ and $A=o(B)$ for $A/B \rightarrow 0$. 

\section{Preliminaries}\label{preliminary}
In this section, we collect some facts concerning the framework of the methodology, consisting of some facts about elliptic curves with complex multiplication and the use of ray class groups to detect elements in the norm group. 
 
\subsection{Elliptic curves with complex multiplication over $K_{\p}$}
Denote by $E_0(K_{\p})$ the set of points of $E(K_{\p})$ with non--singular reduction modulo $\p$. From \cite[\S7, Proposition VII.6.3]{Silverman1}, we know that for $n \geq 1$, there is a subgroup filtration
\begin{equation}\label{filtration}
\cdots \subset E_{2}(K_{\p}) \subset E_{1}(K_{\p}) \subset E_0(K_{\p}), \, E_{n}(K_{\p})=\{P\in E_{0}(K_\p)\colon P\equiv O \bmod \p^n\} \cup \{O\}.
\end{equation}

\begin{defn}
If $P \in E(K_\p)\setminus E_1(K_\p)$, we set $\nu_{\p}(P)=0$. If $P\in E_1(K_\p)$, we define 
\[
\nu_{\p}(P)=\Sup\{i\colon P\in E_i(K_\p)\}. 
\]
\end{defn}

When $E$ has complex multiplication by $\OO$, then in fact the above filtration forms a $\OO$--module filtration. 

\begin{theorem}\label{Streng module}
For any $n\geq 0$, the group $E_{n}(K_\nu)$ is a $\Z[\omega]$--submodule of $E(K_\nu)$. Moreover, for $n \geq 1$, we have an isomorphism of $\OO$--modules
\[
E_{n}(K_{\p})/E_{n+1}(K_{\p})\cong l, 
\]
where $l=\OO_{K_{\p}}/\p$ is the residue field of $K_{\p}$. 
\end{theorem}

\begin{proof}
For $n=0$, see \cite[Theorem 2.13]{Edison1}; for $n\geq 1$, see \cite[Corollary 2.10]{Streng}.
\end{proof}

For the valuation of a multiple of a point, we have the following inequality. 

\begin{lemma}\label{valuation}
Let $P \in E_1(K)$ and $\alpha \in \End(E)$, then $\nu_{\p}\left([\alpha]P\right) \geq \nu_{\p}(P)+\nu_{\p}(\alpha)$, with equality if $\p \nmid \alpha\OO_K$.
\end{lemma}

\begin{proof}
From \cite[Corollary 2.10]{Streng}, we know that there is an isomorphism of $\OO$--modules $E_i(K)/E_{i+1}(K) \cong \OO_K/\p$, so $E_i(K)/E_{i+1}(K)$ is a residue field of size $\Norm(\p)$. If $\nu_{\p}(\alpha\OO_K)=0$, then $[\alpha]$ acts as an automorphism in $\OO_K/\p$. Otherwise, $[\alpha]$ acts as multiplication by $0$ in $\OO_K/\p$. This proves the lemma. 
\end{proof}

\subsection{Ray class group and the norm condition}\label{ray class group and norm}
In this section, we provide a brief explanation on the use of ray class group to determine if an element in a number field can be an ideal norm in a finite abelian extension. Let $\m$ be an integral ideal in $K$. Since we have assumed that $K$ contains an imaginary quadratic field $F$, so $K$ has no real embeddings. Throughout this paper, we use the following notations. 
\begin{itemize}
\item $I_{K}$, the group of fractional ideals in $K$. 
\item $I^{\m}_{K} \subseteq I_{K}$, the subgroup of fractional ideals coprime to $\m$.
\item $P_{K}^{\m} \subseteq I^{\m}_{K}$, the subgroup of principal fractional ideals $(\alpha)$ with the generator $\alpha$ satisfying $\alpha \equiv 1 \bmod \m$. 
\item $Cl_{K}^{\m}=I_{K}^{\m}/P_{K}^{\m}$, the ray class group with modulus $\m$.
\end{itemize}

Recall that for any finite abelian extension of number fields $L/K$, the norm map $\Nm_{L/K}\colon I_L \rightarrow I_K$ is the unique homomorphism such that for any prime ideal $\PP$ of $L$, $\Norm_{L/K}(\PP)=\p^{f(\PP/\p)}$ for $\p$ unramified in $L$, where $f(\PP/\p)=[\OO_L/\PP\colon \OO_K/\p]$ is the degree of the residue field extension. Therefore, if $\A \in N_{L/K}(I_L)$ with $\p \mid \A$, then we necessarily need $\nu_{\p}(\A)$ to be divisible by $f(\PP/\p)$. 

Let $\m$ be a modulus divisible by the conductor $\mathfrak{f}_{L/K}$ of the extension $L/K$, then the global Artin map with respect to $\m$ is 
\begin{align*}
\Art_{L/K}^{\m} \colon I_{K}^{\m} &\longrightarrow \Gal(L/K), \\
\p_1^{n_1} \ldots \p_t^{n_t} &\mapsto \prod_{i=1}^{t} \Frob_{L/K}(\p_i)^{n_i},
\end{align*}
where $ \Frob_{L/K}(\p_i)$ is the unique Frobenius element in $\Gal(L/K)$ that acts as the Frobenius automorphism on the residue field. For $\PP$ lying above $\p$, $ \Frob_{L/K}(\p)$ has order $f(\PP/\p)$ as an element in $\Gal(L/K)$. The kernel of the Artin map is the Takagi group associated to $\m$, $T_{L/K}^{\m} \coloneqq P_{K}^{\m}N_{L/K}(I_{L}^{\m})$. In particular, this implies $N_{L/K}(I_{L}^{\m})\subseteq T_{L/K}^{\m}$ and $P_{K}^{\m} \subseteq T_{L/K}^{\m}$. It is known that the global Artin map is surjective (see \cite[Theorem 8.2]{Cox}), so by the first isomorphism theorem, the Artin map induces an isomorphism between the Galois group $\Gal(L/K)$ and a quotient of the ray class group $Cl_{K}^{\m}$: 
\begin{align*}
I_{K}^{\m}/T_{L/K}^{\m} &\xrightarrow{\sim} \Gal(L/K);\\
 [\A] &\mapsto \Art_{L/K}^{\m} (\A)=\prod_{\p \mid \A}\Frob_{L/K}(\p)^{n_{\p}}.
\end{align*}
Now idenitfy $\Gal(L/K)$ with $I_{K}^{\m}/T_{L/K}^{\m}$, let 
\begin{equation}\label{character}
\chi\colon I_{K}^{\m}/T_{L/K}^{\m}\rightarrow \C^{\times}
\end{equation}
be a group character. Extend $\chi$ to $I_K$ as folllows: for a fractional ideal $\A \in I_{K}$, set  $\chi(\A)=0$ whenever $\A$ is not coprime with $\m$. Now the value of $\chi(\p)$ encodes the behaviour of $ \Frob_{L/K}(\p)$ in $\Gal(L/K)$. If $\chi(\A)=1$ for all group character $\chi$, then $\A$ is an element in the Takagi group $T_{L/K}^{\m}$, a necessary condition for $\A$ to be an element in the norm group $N_{L/K}(I_L^{\m})$. This is exactly the norm detecting condition that we we will use in our elliptic sieve: if $y([\alpha]P) \in K$ is a norm in $L$, then we require $\Ord(\chi(\p)) \mid \nu_\p([\alpha]P)$.
\begin{remark}
One should note that $N_{L/K}(I_L^{\m})$ is a subgroup of the Takagi group $T_{L/K}^{\m}$, therefore the use of group characters to perform the elliptic sieve will result an upper bound that is not sharp. However, we note that this is not the main source of error. 
\end{remark}
\section{Elliptic divisibility sequences for CM curve}\label{CM EDS periodicity section}
\subsection{Elliptic divisibility sequences and elliptic denominator sequences}\label{EDS for CM curves}
Let $E/K$ be an elliptic curve given by a Weierstrass equation (\ref{EC}) with coefficients $a_i \in \OO_K$. Let $P=(x, y)$ be a non-torsion point on $E(K)$, $n$ an integer. It is known that there exist rational functions $\psi_n$ on $E$ satisfying 
\begin{equation}\label{expression}
x(nP)=\frac{x(P)\psi_n^2-\psi_{n-1}\psi_{n+1}}{\psi_{n}^{2}},
\end{equation}
where $\psi_n$ is called the $n$-division polynomial. One important property that is relevant for us is that the $n$--division polynomial satisfies the recurrence relation 
\begin{equation}\label{recurrence}
\psi_{n+m}\psi_{n-m}\psi_r^2 = \psi_{m+r}\psi_{m-r}\psi_n^2 -\psi_{n+r}\psi_{n-r}\psi_m^2  \text{ for any integers } n>m>r,
\end{equation}
as shown in \cite[Exercise 3.7(c)]{Silverman1}. The sequence $\psi_n$ forms a divisibility sequence in $\Z[x, y]$, that is, $\psi_n \mid \psi_m$ for $n \mid m$. For more details and basic properties of $\psi_n$, see \cite[Chapter 2]{Lang}, or \cite[Section 1]{Akbary} for a modern exposition. 

If $E/K$ admits complex multiplication by an order $\OO$ in an imaginary quadratic field $F \subseteq K$, there are existing generalisations of division polynomials that takes the extra endomorphisms $\alpha \in \OO\setminus \Z$ into consideration. For example, in \cite{Satoh} Satoh proved the existence of generalised division polynomials for elliptic curves with complex multiplication, while Stange extended the definition of division polynomials from being attached to endomorphisms to arbitrary isogenies in \cite{Stange4}; both authors demonstrated that these polynomials satisfy certain recurrence relations subject to certain restrictions. For our purpose, we use the generalisation defined in \cite{Edison1}, which uses the framework of elliptic nets and therefore free from those restrictions; the author of \cite{Edison1} named them CM division polynomials. Elliptic nets are higher rank generalisation of elliptic divisibility sequences, in particular they are used to study $\Z$--linear combination of points on elliptic curves. We refer interested readers to \cite{Stange1} for more details. For $z \in \OO=\Z[\tau]$, the CM division polynomial satisfies 
\begin{equation}\label{expression2}
x([z]P)=\frac{x(P)\psi_{z}^2-\psi_{z-1}\psi_{z+1}}{\psi_{z}^{2}}=\frac{x\left([\tau]P\right)\psi_{z}^2-\psi_{z-\tau}\psi_{z+\tau}}{\psi_{z}^{2}},
\end{equation}
and the recurrence relation 
\begin{equation}\label{CM recurrence}
\psi_{\alpha+\beta}\psi_{\alpha-\beta}\psi_{\gamma}^2=\psi_{\alpha+\gamma}\psi_{\alpha-\gamma}\psi_{\beta}^2-\psi_{\beta+\gamma}\psi_{\beta-\gamma}\psi_{\alpha}^2 \text{ for any } \alpha, \beta, \gamma \in \OO.
\end{equation}

Regardless of whether the curve $E$ admits complex mutliplication, in general there will be some common factors between the numerator and the denominator (ideals) in (\ref{expression}), so the expression will not be in lowest terms. For our purpose, we want to work with the genuine denominator of a point. 

\begin{defn}\label{denominator sequence def}
Let $P=(x, y)$ be a non-torsion point in $E(K)$. For $\alpha \in \OO=\End(E)$, we define the sequence of ideals $D_{\alpha}$ in $\OO_K$, indexed by elements of $\OO$, by 
\begin{equation}
x\left([\alpha]P\right)\OO_{K}=A_{\alpha}D_{\alpha}^{-2},
\end{equation}
where $A_{\alpha}$ and $B_{\alpha}$ are coprime $\OO_{K}$--ideals. If $K$ has class number $1$, we instead define the sequence of elements $B_{\alpha}$, indexed by elements of $\OO$, where $D_{\alpha}=\left(B_{\alpha}\right)$. We will call $\left(D_{\alpha}\right)_{\alpha \in \OO}$ and $\left(B_{\alpha}\right)_{\alpha \in \OO}$ elliptic denominator sequences (of ideals and elements respectively). 
\end{defn}

The main reason we make the definition above is to avoid ambiguity in the terminology EDS. One way of defining an EDS in a commutative ring would be a divisibility sequence satisfying (\ref{recurrence}), while another closely related (but not exactly equivalent) definition of EDS is the sequence of denominators of $\left([\alpha]P\right)_{\alpha \in \End(E)}$, which we have termed the elliptic denominator sequence $\left(D_{\alpha}\right)_{\alpha \in \OO}$ above. An elliptic denominator sequence is generally not an EDS, as it need not satisfy the recurrence relation (\ref{CM recurrence}). In \cite[Example A.2]{Edison1}, the author showed that differences can occur if $P$ admits primes of bad reduction. However, subject to a condition on the indices, the author also showed in \cite[Theorem 1.7]{Edison1} that elliptic denominator sequences do satisfy a weakened version of (\ref{CM recurrence}).

\begin{prop}\label{CM EDS}
Let $M$ be an element in the ideal
\[
\mathfrak{M}(P)=\Lcm\{\Ann(P \bmod \p) \colon \p \text{ prime }\},
\]
where $\Ann(P \bmod \p)$ is the annihilator of $P$ in the $\OO$--module $E(K_{\p})/E_{0}(K_{\p})$. Let $\alpha, \beta$ and $\gamma$ be elements of $\OO=\End(E)$.
\begin{enumerate}
\item If $\OO$ is a principal ideal domain and two of $\alpha, \beta$ and $\gamma$ are multiples of $M$, the generator of $\mathfrak{M}(P)=(M)$, then the sequence of elements $\left(B_{\alpha}\right)_{\alpha \in \OO}$ can be chosen to satisfy
\begin{equation}
B_{\alpha+\beta}B_{\alpha-\beta}B_{\gamma}^2=B_{\alpha+\gamma}B_{\alpha-\gamma}B_{\beta}^2-B_{\beta+\gamma}B_{\beta-\gamma}B_{\alpha}^2.
\end{equation}
\item If $\OO$ is not a principal ideal domain and each of $\alpha, \beta$ and $\gamma$ is either a multiple of $M$, or coprime to $M$, then there is a sequence of ideals $\{\mathfrak{g}_{\alpha}\}_{\alpha \in \OO}$ in $\OO_K$ with 
\begin{align}
D_{\alpha}&=\psi_{\alpha}\mathfrak{g}_{\alpha}^{-1}, \\
\psi_{\alpha+\beta}\psi_{\alpha-\beta}\psi_{\gamma}^2 &= \psi_{\alpha+\gamma}\psi_{\alpha-\gamma}\psi_{\beta}^2 - \psi_{\beta+\gamma}\psi_{\beta-\gamma}\psi_{\alpha}^2, \\
\mathfrak{g}_{\alpha+\beta}\mathfrak{g}_{\alpha-\beta}\mathfrak{g}_{\gamma}^2 &= \mathfrak{g}_{\alpha+\gamma}\mathfrak{g}_{\alpha-\gamma}\mathfrak{g}_{\beta}^2 = \mathfrak{g}_{\beta+\gamma}\mathfrak{g}_{\beta-\gamma}\mathfrak{g}_{\alpha}^2.
\end{align}
\end{enumerate}
\end{prop}

%It is worth noting that the CM division polynomials used in this paper are odd functions: they satisfy $\psi_{-\alpha}=-\psi_{\alpha}$  (see \cite[Proposition 2.1]{Stange1} for proof). Therefore, if $D_{\alpha}$ is a principal ideal with generator $B_{\alpha}$, we use the convention $B_{-\alpha}=-B_{\alpha}$ and permute the variables if necessary.
In comparison with division polynomials, elliptic denominator sequences have better $\p$--adic properties. Furthermore, they are also divisibiltiy sequences. Before we state the divisibility properties, we first note that we did not assume $\OO$ to be a principal ideal domain, therefore we need to extend the elliptic denominator sequence from an $\OO$--index seqeunce to an ideal--indexed sequence. This is due to Streng in \cite[Section 3]{Streng}, who actually first generalised elliptic denominator sequences (although he termed them elliptic divisibility sequences in \cite{Streng}) to curves with complex multiplication, in order to generalise a result concerning primitive divisors of EDS. 

\begin{defn}
Let $P=(x, y)$ be a non-torsion point in $E(K)$. Let $\A$ be an ideal in $\OO$. An elliptic denominator sequence indexed by ideals $\A$ of $\OO$, associated to $P$, is given by 
\[
D_{\A}=\sum_{\alpha \in \A}D_{\alpha}.
\]
\end{defn}

\begin{prop}[Divisibility properties] \label{divisibility}
Let $\alpha, \beta$ be elements in $\OO$, $\mathfrak{a}, \mathfrak{b}$ be ideals in $\OO$ and $\alpha \in \mathfrak{a}, \beta \in \mathfrak{b}$.
\begin{enumerate}[(i)]
\item Weak divisibility: for all $\alpha, \beta \in \OO$, if $\alpha \mid \beta$, then $D_{\alpha} \mid D_{\beta}$.
\item Strong divisibility: $D_{\mathfrak{a}+ \mathfrak{b}}=D_{\mathfrak{a}}+ D_\mathfrak{b}$. 
\end{enumerate}
\end{prop}

\begin{proof}
See \cite[Lemmata 3.2, 3.3]{Streng}.
\end{proof}

\begin{remark}
By weak divisibility, we have $D_{\alpha}=D_{\alpha\OO}$ for all $\alpha \in \OO$. It is also worth emphasising that when we state the strong divisibility property, one must refer to an ideal--indexed sequence. 
\end{remark}

\subsection{Periodicity of CM EDS}
In this section, we show that the sequence of ray classes corresponding to the denominator ideals $D_{\alpha}$ is periodic in ray class groups with respect to arbitrary modulus $\m$. If the elliptic curve $E/K$ is defined over a class number $1$ field, then we can show that the sequence $B_{\alpha}$ is periodic modulo any prime power and hence any element in $\OO_K$. The proofs in both cases share a similar structure; we first show that the sequence modulo a prime (resp. in ray class) forms a periodic sequence of a certain form, then conclude the proof of periodicity by finding an upper bound of the period and with the use of the Chinese Remainder theorem. For clarity purpose, we separate the proofs of the two cases. 

\subsubsection{Imaginary fields of class number $1$}
\begin{prop}[Symmetry Law]\label{symmetry law}
Let $M$ be as in \Cref{CM EDS}. Let $\alpha, r \in \OO$ with $r$ being a multiple of $M$. Let $\p$ be a prime in $\OO_K$ and $k \in \N$. Suppose that $K$ is a class number $1$ number field and $\p^k$ divides $B_{r}/\Gcd(B_r, B_M)$, then for all $x+y\tau \in \OO=\Z[\tau]$,
\[
B_{\alpha+(x+y\tau)r}\equiv \begin{cases}
			 U^{\frac{x(x-1)}{2}}V^{\frac{y(y-1)}{2}}
\left(\frac{B_{\alpha+r}}{B_\alpha}\right)^{x}
\left(\frac{B_{\alpha+r\tau}}{B_\alpha}\right)^{y}
\left(\frac{B_{\alpha}B_{\alpha+(1+\tau)r}}
{B_{\alpha+r}B_{\alpha+r\tau}}\right)^{xy}
B_{\alpha} \bmod \p^k, & \text{if $\p^k \nmid B_\alpha$}\\
            0 \bmod \p, & \text{if $\p^k \mid B_{\alpha}$},
		 \end{cases}
\]
where
\[
U=\frac{B_{M+r}B_{M-r}}{B_M^2},
\qquad
V=\frac{B_{M+r\tau}B_{M-r\tau}}{B_M^2}
\]
and all the quotients in the first case are $\p$--adic units. 
\end{prop}

\begin{proof}
Note that for any $x+y\tau \in \OO$, we have $((\alpha)+(x+y\tau)r)+(r)=(\alpha)+(r)$, so strong divisibility (\Cref{divisibility}) immediately implies that 
\begin{equation}\label{gcdd}
\gcd(B_{\alpha}, B_{r})=\gcd(B_{\alpha+(x+y\tau)r}, B_{r})
\end{equation}
for every $x+y\tau \in \OO$. 

Therefore, if $\p^k$ divides $B_{\alpha}$, then we are done, so assume not. 

It is obvious that $r \vert r\tau$, so for $a \equiv b \bmod r\tau$, we have $a \equiv b \bmod B_{r\tau}$ and hence $a\equiv b \bmod B_r$.

Using the recurrence relation and if we replace $m$ by $M$ and $n$ by $\alpha+(x+y\tau)r$, we obtain
\begin{equation}\label{RR1}
B_{M+r}B_{M-r}B_{\alpha+(x+y\tau)r}^2 \equiv B_{\alpha+(x+1+y\tau)r}B_{\alpha+(x-1+y\tau)r}B_{M}^2 \bmod B_{r}^2.
\end{equation}
Similarly, replacing $r$ by $r\tau$, $m$ by $M$ and $\alpha$ by $\alpha+(x+y\tau)$, we also obtain
\begin{equation}\label{RR2}
B_{M+r\tau}B_{M-r\tau}B_{\alpha+(x+y\tau)r}^2 \equiv B_{\alpha+(x+(1+y)\tau)r}B_{\alpha+(x+(y-1)\tau)r}B_{M}^2 \bmod B_{r}^2.
\end{equation}
Let 
\begin{equation}\label{a_l}
U=\frac{B_{M+r}B_{M-r}}{B_M^2}, V=\frac{B_{M+r\tau}B_{M-r\tau}}{B_M^2}, a_{x+y\tau}=\frac{B_{\alpha+(x+y\tau)r}}{\gcd(B_{\alpha}, B_{r})}.
\end{equation}

By weak divisibility (\Cref{divisibility}), the assumption of $r$ being a multiple of $M$ implies that $U$ and $V$ are both elements in $\OO_K$. Recall from (\ref{gcdd}) that $\gcd(B_{\alpha}, B_{r})=\gcd(B_{\alpha+(x+y\tau)r}, B_{r})$, so $a_{x+y\tau}$ is an element in $\OO_K$ coprime with $B_r/\gcd(B_{\alpha+(x+y\tau)r}, B_{r})$ and this allows us to divide (\ref{RR1}) by $B_{\alpha+(x+y\tau)r}^2B_{M}^2$, we get 
\[
U \equiv \frac{a_{x+1+y\tau}a_{x-1+y\tau}}{a_{x+y\tau}^2} \bmod \frac{B_r^2}{\gcd(B_{\alpha}B_{M}, B_r)^2},
\]
where every $a_{x+y\tau}$ is coprime to the modulus. If we now fix the value of $y$ and induct on $x$, then we obtain 
\[
a_{x+y\tau}\equiv U^{\frac{x(x-1)}{2}}a_{1+y\tau}^xa_{y\tau}^{1-x} \bmod \frac{B_r^2}{\gcd(B_{\alpha}B_{M}, B_r)^2}.
\]

We now do the same for (\ref{RR2}): divide both sides by $B_{\alpha+(x+y\tau)r}^2B_{M}^2$ to obtain
\[
V \equiv \frac{a_{x+(1+y)\tau}a_{x+(y-1)\tau}}{a_{x+y\tau}^2} \bmod \frac{B_r^2}{\gcd(B_{\alpha}B_{M}, B_r)^2},
\]
then we fix the value of $x$ and it follows from induction on $y$ that 
\begin{align*}
a_{x+y\tau}
&\equiv V^{\frac{y(y-1)}{2}} a_{x+\tau}^{y} a_{x}^{1-y} \\
&\equiv V^{\frac{y(y-1)}{2}}
\left(
U^{\frac{x(x-1)}{2}} a_{1+\tau}^{x} a_{\tau}^{1-x}
\right)^{y}
\left(
U^{\frac{x(x-1)}{2}} a_{1}^{x} a_{0}^{1-x}
\right)^{1-y} \\
&\equiv
U^{\frac{x(x-1)}{2}}
V^{\frac{y(y-1)}{2}}
\left(\frac{a_{1}}{a_{0}}\right)^{x}
\left(\frac{a_{\tau}}{a_{0}}\right)^{y}
\left(\frac{a_{1+\tau} a_{0}}{a_{1} a_{\tau}}\right)^{xy}
a_{0}
\bmod \frac{B_r^2}{\gcd(B_{\alpha} B_{M}, B_r)^2}.
\end{align*}
Multiplying both sides by $\Gcd(B_\alpha, B_r)$, we obtain 
\begin{align*}
a_{x+y\tau}\Gcd(B_{\alpha}, B_r)
&\equiv 
\Gcd(B_{\alpha}, B_r)C^{\frac{x(x-1)}{2}}
D^{\frac{y(y-1)}{2}} \\
& \quad \cdot
\left(\frac{a_{1}}{a_{0}}\right)^{x}
\left(\frac{a_{\tau}}{a_{0}}\right)^{y}
\left(\frac{a_{1+\tau} a_{0}}{a_{1} a_{\tau}}\right)^{xy}
a_{0}
\bmod \frac{\Gcd(B_{\alpha}, B_r)B_r^2}{\gcd(B_{\alpha} B_{M}, B_r)^2}.
\end{align*}

By assumption, $\p^k$ divides $B_{r}/\Gcd(B_r, B_M)$. Since $B_{r}/\Gcd(B_r, B_M)$ divides the modulus clearly, so as $\p^k$. Now apply the definition of $a_{x+y\tau}$ from (\ref{a_l}), then the first case in the proposition is proved. Furthermore, recall that $a_{x+y\tau}$ is coprime with $B_r/\gcd(B_{\alpha+(x+y\tau)r}, B_{r})$ and therefore coprime to the modulus as well. Therefore, $U$, $V$, $B_{\alpha +r}B_{\alpha}^{-1}=a_1/a_{0}^{-1}$, $B_{\alpha + r\tau}B_{\alpha}^{-1}=a_{\tau}a_0^{-1}$ and $B_{1+\tau}B_{\alpha}\left(B_{1}B_{\tau}\right)^{-1}a_{1+\tau}a_{0}a_{1}^{-1} a_{\tau}^{-1}$ are all $\p$--adic units. 
\end{proof}

\begin{prop}[Periodicity] \label{PID periodicity}
Suppose $K$ is a class number $1$ number field. Let $M$ be as in \Cref{CM EDS}, let $k \in \N$ and let $\p$ be a prime ideal in $O_K$. Let 
\[
r(\p^k)=M \Ord \left([M]P \bmod \p^{k+\nu_{\p}([M]P)}\right),
\]
and 
\[
\pi(\p^k)=\phi(\p^k)r(\p^k)=(\Nm(\p)-1)\Nm(\p^{k-1})r(\p^k),
\] 
where $\phi(\A)=\#(\OO_K/\A)^{\times}=\prod_{\p^k \mid\!\mid \A}(\Nm(\p)-1)\Nm(\p^{k-1})$ is the generalisation of the Euler totient function. If $x+y\tau \in \pi(\p^k)\OO$, or equivalently $\pi(\p^k)$ divides both $x$ and $y$, then we have 
\[
B_{\alpha + (x+y\tau)r(\p^k)}\equiv B_{\alpha} \bmod \p^k.
\]
In other words, the sequence $B_{\alpha} \bmod \p^k$ is periodic with period dividing $\pi(\p^k)$. 
\end{prop}

\begin{proof}
Write $r=r(\p^k)$ throughout the proof. Suppose $\p^k$ does not divide $B_\alpha$. By assumption, we have 
\[
[r]P \equiv O \bmod \p^{k+\nu_\p(MP)}
\]
and therefore $k+\nu_{\p}(MP)\leq \nu_{\p}(rP)$. For all prime ideals $\p$, we have 
\begin{align*}
\nu_\p(B_r)
&= \nu_{\p}(rP) \\
&\geq k + \nu_\p(MP) \\
&= k+\nu_{\p}(B_M).
\end{align*}
Therefore,
\[
\p^k \mid \frac{B_r}{\Gcd(B_M,B_r)} 
\quad \text{and} \quad 
M \mid r.
\]

Recall from \Cref{symmetry law} that if $\p^k \nmid B_{\alpha}$, then for every 
$x+y\tau \in \End(E)$,
\[
B_{\alpha+(x+y\tau)r}
\equiv 
U^{\frac{x(x-1)}{2}}
V^{\frac{y(y-1)}{2}}
\left(\frac{B_{\alpha+r}}{B_\alpha}\right)^{x}
\left(\frac{B_{\alpha+r\tau}}{B_\alpha}\right)^{y}
\left(\frac{B_{\alpha}B_{\alpha+(1+\tau)r}}{B_{\alpha+r}B_{\alpha+r\tau}}\right)^{xy}
B_{\alpha}
\bmod{\p^k},
\]
where 
\[
U=\frac{B_{M+r}B_{M-r}}{B_M^2}, 
\qquad 
V=\frac{B_{M+r\tau}B_{M-r\tau}}{B_M^2}.
\]

Note that \Cref{symmetry law} also showed that $C$ and $D$ are $\p$--adic units. 
In particular, by strong divisibility,
\[
\Gcd(B_{\alpha}, \p^k)
= \Gcd(B_{\alpha+(x+y\tau)r}, \p^k)
= \Gcd(B_{\alpha}, B_r, \p^k).
\]
Now
\[
\#(\OO_{K}/\p^k)=\phi(\p^k)=(\Nm(\p)-1)\Nm(\p^{k-1}),
\]
so for any $u \in (\OO_{K}/\p^k)^{\times}$, if $\Nm(\p^k)\mid x$ and $\Nm(\p^k)\mid y$, then
\[
u^x \equiv 
u^y \equiv 
u^{\frac{x(x-1)}{2}} \equiv 
u^{\frac{y(y-1)}{2}} \equiv 
u^{xy} \equiv 1 
\bmod{\p^k}.
\]

Thus by definition of $\pi(\p^k)$, if $\pi(\p^k)$ divides $(x+y\tau)r$, then 
\[
B_{\alpha + (x+y\tau)r}\equiv B_{\alpha} \bmod \p^k.
\]
\qedhere
\end{proof}
\vspace{-0.5em}
\subsubsection{Imaginary fields of class number greater than $1$}
In this case, we will first need a lemma that allows us to integrate the recurrence relation (\ref{CM recurrence}) into the ray class group framework. Throughout this section, we assume $\m$ to be a modulus divisible by the conductor $\mathfrak{f}_{L/K}$ of the extension $L/K$, so we have $\Gal(L/K)\cong I_K^{\m}/T_{L/K}^{\m}$. To show that different fractional ideals represent the same class in $I_K^{\m}/T_{L/K}^{\m}$, we use the notation $I\sim J$ whenever $[I]=[J]$ for fractional ideals $I, J \in I_K^{\m}$. 

\begin{lemma}\label{ray class}
Let $A, B, C$ be elements in a number field $K$ that satisfy $A=B+C$. Let $\m$ be an integral ideal in $\OO_K$ and consider the group $I_K^\m / T_{L/K}^{\m}$. Suppose $\g$ is a fractional ideal in $K$ such that $B\g$ is an ideal contained in $\m$ and at least one of $A\g$ and $C\g$ is relatively prime to $\m$ , then $A\g \sim C\g$. 
\end{lemma}

\begin{proof}
We were given $A=B+C$, so without loss of generality assume $C\neq 0$ and divide $C$ on both sides to get 
\[
\frac{A}{C}=\frac{B+C}{C}=\frac{B}{C}+1.
\]
Since $B\g \subseteq \m$, equivalently $\nu_{\p}(B\g)\geq \nu_{\p}(\m)>0$ for every prime $\p \mid \m$. We also assumed that at least one of $A\g$ and $C\g$ is relatively prime to $\m$; without loss of generality take $\nu_{\p}(C\g)=0$, so we may deduce that for every $\p\mid\m$,
\[
\nu_{\p}\left(BC^{-1}\right)=\nu_{\p}(B\g)-\nu_{\p}(C\g) \geq \nu_{\p}(\m) \implies BC^{-1} \equiv 0 \bmod \m. 
\]
Note that $(A\g)(C\g)^{-1}=(AC^{-1})$ as ideals, therefore it follows that $AC^{-1}=A\g(C\g)^{-1}=(\alpha)$ with $\alpha \equiv 1 \bmod \m$. Equivalently, $AC^{-1} \in P_{K}^{\m}$, hence $A\g$ and $C\g$ represent the same ray class in $I_K^{\m}/P_K^{\m}$. Since we have assumed that $\mathfrak{f}_{L/K}\mid\m$, we have $P_K^{\m}\subseteq T_{L/K}^{\m}$ and therefore $A\g \sim C\g$. 
\end{proof}

\begin{prop}\label{Ray class symmetry}
Let $K$ be a number field of class number greater than $1$. Let $M$ be as in \Cref{CM EDS}. Let $\alpha, r \in \OO$ with $r$ being a multiple of $M$. Let $\p$ be a prime in $\OO_K$ and $k \in \N$; denote $\m = \p^k$ and consider the group $I_{K}^{\m}/T_{L/K}^{\m}$. Suppose $\p^k$ is a factor of the integral ideal $D_{r}\left(D_r+D_M\right)^{-1}$ and $\p + D_{\alpha}=\OO_K$, then for all $x+y\tau \in \OO$, 
\begin{equation}\label{ray class congruence}
\begin{split}
D_{\alpha+(x+y\tau)r} &\sim 
U^{\frac{x(x-1)}{2}}
V^{\frac{y(y-1)}{2}}
\left(D_{\alpha+r}D_{\alpha}^{-1}\right)^{x} \\
&\quad \cdot \left(D_{\alpha+r\tau}D_{\alpha}^{-1}\right)^{y}
\left(D_{\alpha+(1+r)\tau}D_{\alpha}
\left(D_{\alpha+r}D_{\alpha+r\tau}\right)^{-1}\right)^{xy} \\
&\quad \cdot D_{\alpha}.
\end{split}
\end{equation}
\end{prop}

\begin{proof}
For number fields $K$ with class number greater than $1$, \Cref{CM EDS} tells us that there is a sequence of fractional ideals $\{\mathfrak{g}_{\alpha}\}_{\alpha \in \OO}$ in $\OO_K$ with 
\begin{align}
D_{\alpha}&=\psi_{\alpha}\mathfrak{g}_{\alpha}^{-1}, \\
\psi_{\alpha+\beta}\psi_{\alpha-\beta}\psi_{\gamma}^2 &= \psi_{\alpha+\gamma}\psi_{\alpha-\gamma}\psi_{\beta}^2 - \psi_{\beta+\gamma}\psi_{\beta-\gamma}\psi_{\alpha}^2, \\
\mathfrak{g}_{\alpha+\beta}\mathfrak{g}_{\alpha-\beta}\mathfrak{g}_{\gamma}^2 &= \mathfrak{g}_{\alpha+\gamma}\mathfrak{g}_{\alpha-\gamma}\mathfrak{g}_{\beta}^2 = \mathfrak{g}_{\beta+\gamma}\mathfrak{g}_{\beta-\gamma}\mathfrak{g}_{\alpha}^2.
\end{align}

As in the case of $\Cl(K)=1$, we use the recurrence relation and replace the indices $\beta$, $\gamma$ by $M$ and $\alpha + (x+y\tau)r$ to obtain 
\begin{equation}\label{NON_PID_RR}
\psi_{M+r}\psi_{M-r}\psi_{\alpha+(x+y\omega)r}^2 = \psi_{\alpha+(x+1+y\omega)r}\psi_{n+(x-1+y\omega)r}\psi_{M}^2 + \psi_{\alpha+(x+y\omega)r+M}\psi_{\alpha+(x+y\omega)r-M}\psi_r^2.
\end{equation}

Next, as elements in $K$, divide both sides by $\psi_M^2\psi_{\alpha+(x+y\omega)r}^2$ to obtain 
\begin{align*}
\psi_{M+r}\psi_{M-r}\psi_{M}^{-2} 
&= \psi_{\alpha+(x+1+y\omega)r}\psi_{\alpha+(x-1+y\omega)r}\psi_{\alpha+(x+y\omega)r}^{-2} \\
& + \psi_{\alpha+(x+y\omega)r+M}\psi_{\alpha+(x+y\omega)r-M}\frac{\psi_r^2}{\psi_M^2}\psi_{\alpha+(x+y\omega)r}^{-2}.
\end{align*}

By \Cref{CM EDS}, we also have a sequence of fractional ideals that satisfy 
\begin{align*}
\mathfrak{g}_{M+r}\mathfrak{g}_{M-r}\mathfrak{g}_{M}^{-2} 
&= \mathfrak{g}_{\alpha+(x+y\omega)r+M}\mathfrak{g}_{\alpha+(x+y\omega)r-M}\g_{r}^2\left(\mathfrak{g}_{M}\g_{\alpha+(x+y\omega)r}\right)^{-2} \\
&=\g_{\alpha+(x+1+y\omega)r}\g_{n+(x-1+y\omega)r}\g_{\alpha+(x+y\omega)r}^{-2}.
\end{align*}

Now we wish to apply \Cref{ray class}, then we can multiply $\left(\mathfrak{g}_{M+r}\mathfrak{g}_{M-r}\right)^{-1}\mathfrak{g}_{M}^{2}$ on both sides of (\ref{NON_PID_RR}) and yield the following congruence:
\begin{equation}\label{DA}
D_{M+r}D_{M-r}D_{M}^{-2}\sim D_{\alpha+(x+1+y\omega)r}D_{\alpha+(x-1+y\omega)r}D_{\alpha+(x+y\omega)r}^{-2}.
\end{equation}

To do so, we need to show that $D_{M+r}D_{M-r}D_{M}^{-2}$ and $D_{\alpha+(x+1+y\omega)r}D_{\alpha+(x-1+y\omega)r}D_{\alpha+(x+y\omega)r}^{-2}$ are ideals coprime to $\m$. Indeed, by strong divisibility of $\left(D_{\alpha}\right)_{\alpha \in \OO}$ and the fact that $D_{\alpha}=D_{\alpha\OO}$ for all $\alpha \in \OO$, we know that for all $\gamma \in \OO$, 
\[
D_{\alpha+\gamma r}+D_{r}=D_{(\alpha+\gamma r)+(r)}=D_{(\alpha)+(r)}=D_{\alpha}+D_{r}.
\]
We also assumed that $\p$ is not a prime ideal factor of $D_{\alpha}$. Therefore, $D_{\alpha+\gamma r}$ is an integral ideal coprime to $\p$ for all $\gamma \in \OO$ and hence (\ref{DA}) follows from \Cref{ray class} and (\ref{NON_PID_RR}). 

Repeating the same argument, we also obtain 
\[
D_{M+r\tau}D_{M-r\tau}D_{M}^{-2} \sim D_{\alpha+(x+(y+1)\tau)r}D_{\alpha+(x+(y-1)\tau)r}D_{\alpha+(x+y\tau)r}^{-2}. 
\]

Similar to the case of $Cl(K)=1$, we set 
\begin{equation}
U=\frac{D_{M+r}D_{M-r}}{D_M^2}, V=\frac{D_{M+r\omega}D_{M-r\omega}}{D_M^2}, a_{x+y\omega}=\frac{D_{\alpha+(x+y\omega)r}}{D_{\alpha}+D_{r}}.
\end{equation}
Apply the induction argument, we obtain 
\[
a_{x+y\omega} \sim
U^{\frac{x(x-1)}{2}}
V^{\frac{y(y-1)}{2}}
\left(a_1a_0^{-1}\right)^{x}
\left(a_{\omega}a_{0}^{-1}\right)^{y}
\left(a_{1+\omega}a_{0}\left(a_{1}a_{\omega}\right)^{-1}\right)^{xy}
a_{0}.
\]

Since $D_{\alpha}+D_{r}$ is coprime to the modulus $\m=\p^k$, $[D_{\alpha}+D_{r}]$ is a well--defined ray class and we can multiply it on both sides to conclude the proof. 
\end{proof}

\begin{cor}[Periodicity]\label{Non PID periodicity}
Suppose $K$ is a number field of class number greater than $1$. Let $M$ be as in \Cref{CM EDS}, let $k \in \N$ and let $\p$ be a prime ideal in $O_K$. Let 
\[
r(\p^k)=M \Ord \left([M]P \bmod \p^{k+\nu_{\p}([M]P)}\right).
\]
Assume $D_{\alpha}$ is coprime to the prime ideal $\p$ and consider the group $I_{K}^{\m}/T_{L/K}^{\m}$, where $\m$ is the ideal $\p^k$. For every $\beta \in \OO$, if $\beta \equiv \alpha \bmod \pi(\p^k)$, where 
\[
\pi(\p^k)=r(\p^k)h^{\m}(K), \text{ $h^{\m}(K)$ is the ray class number,}
\] 
then the denominator ideals satisfy $D_{\beta}\sim D_{\alpha}$. %In other words the sequence of ideals $D_{\alpha}$ is periodic in the group $I_K^{\m}/T_{L/K}^{\m}$ with period dividing a multiple of the ray class number. 
\end{cor}

\begin{remark}
In \Cref{CM EDS}, we stated that $M$is an element in $\OO$, while $h^{\m}(K)$ is clearly a natural number. Hence, $\pi(\p^k)$ is an element in $\OO$ for any $\OO_K$--prime $\p$. 
\end{remark}

\begin{proof}
We first need to show that $\p^k$ does divide the integral ideal $D_r(D_r+D_M)^{-1}$, whose proof is the exact same as the one in \Cref{PID periodicity}. Since $I_K^{\m}/T_{L/K}^{\m}$ is just a quotient of the ray class group $\Cl_{K}^{\m}$, \Cref{Ray class symmetry} or (\ref{ray class congruence}) immediately tells us that for any $x+y\tau \in \OO$, if both $x$ and $y$ are divisible by $h^{\m}(K)$ (so $(x+y\tau)r(\p^k)$ is divisible by $\pi(\p^k)$), then we obtain 
\[
D_{\alpha + (x+y\tau)r(\p^k)}\sim D_{\alpha}.
\]
Write $\beta = \alpha + (x+y\tau)r(\p^k)$ to conclude the proof. 
\end{proof}

Now suppose $D_{\alpha} \sim D_{\beta}$. Unlike the case when $K$ has class number $1$, we cannot apply the Chinese Remainder theorem directly to ray classes to show periodicity in the ray class group for arbitrary modulus $\m$. However, we can apply the Chinese Remainder theorem to the congruence conditions arising from the prime-power factors of \(\mathfrak m\), therefore obtain a common period and get the claim we want. Let $\p_1^{k_1}\ldots\p_n^{k_n}$ be the unique ideal factorisation of $\m$ (note that $K$ is assumed to be imaginary, so there is no issue about real places), then the corollary above tells us that there exists $\delta_i \in K$ and $\mathfrak{d}_{i} \in N_{L/K}(I_L)$ such that 
\[
\alpha \equiv \beta \bmod \pi(\p_i^{k_i}) \implies D_{\alpha}D_{\beta}^{-1}=(\delta_i)\mathfrak{d}_{i}, \delta_i \equiv 1 \bmod \p_i^{k_i}, \mathfrak{d}_{i} \in N_{L/K}\left(I_{L}^{\p_i^{k_i}}\right).
\]
 for all $i=1, \ldots, n$. Therefore, we can apply the Chinese Remainder theorem to claim that for any arbitrary modulus $\m$, there exists $\delta \in K$ and $\mathfrak{d} \in N_{L/K}(I_L)$ such that 
\[
\alpha \equiv \beta \bmod \bigcap\limits_{i=1}^{n}\left(\pi(\p_i^{k_i})\right) \implies D_{\alpha}D_{\beta}^{-1}=(\delta)\mathfrak{d}, \delta \equiv 1 \bmod \m, 
\]
\[
\text{where } \delta=\prod_{i=1}^n \delta_i \text{ and }\mathfrak{d}=\prod_{i=1}^{n}\mathfrak{d}_i \in N_{L/K}\left(I_L^{\m}\right). 
\]

This shows that the sequence of ray classes corresponding to the denominator ideals $(D_{\alpha})_{\alpha\in \mathcal{O}}$ is periodic in $I_K^{\m}/T_{L/K}^{\m}$ for arbitrary modulus $\mathfrak{m}$ coprime to $D_{\alpha}$. Consequently, we can establish the existence of infinitely many elements in the sequence $(D_{\alpha})$, indexed by $\mathcal{O}$-prime elements, that are divisible by $\mathcal{O}_K$-primes up to a specified valuation and are non-trivial with respect to a given ray class character. Unlike the original proof in \cite[Section 3.5]{Brauer}, restricting the base fields of the elliptic curves under consideration to imaginary fields avoids complications with the signs of denominator sequences, thereby significantly simplifying the proof.

\begin{prop}\label{Lambda estimate}
Let $\chi \colon I_K^{\m}/T_{L/K}^{\m} \rightarrow \C^{\times}$ be a group character with modulus $\m$ as described in (\ref{character}), where $\m$ is divisible by all the primes $\p \in \OO_K$ that ramify in $L$. Denote the conductor of the order $\OO$ by $f=[\OO_F\colon \OO]$. Let $\pi$ be as defined in \Cref{Non PID periodicity} for the sequence $(D_{\alpha})_{\alpha \in \OO}$ (resp.\ \Cref{PID periodicity} for $(B_{\alpha})_{\alpha \in \OO}$) in the group $I_K^{\m}/T_{L/K}^{\m}$ (resp.\ modulo $\m$). Suppose that there exists $\alpha \in \OO$ such that $\alpha \in (\OO/\pi\OO)^{\times}$ and the character $\chi$ satisfies $\chi(D_{\alpha})\not\in \{0, 1\}$ (resp.\ $\chi(B_{\alpha})\not\in \{0, 1\}$), then the set 
\begin{multline}
\{\ell \in \OO\colon \ell \text{ generates a prime ideal in } \OO \text{ and coprime to } f,  \Ord(\chi(\p)) \nmid \nu_{\p}(D_{\ell}) \text{ and } \p \nmid \m\\
\text{for some prime } \p \subset \OO_K \text{ with } \Norm(\ell)\le X\} \sim \frac{1}{\#(\OO/\pi\OO)^{\times}}\cdot\frac{1}{h(\OO)}\cdot\frac{X}{\log X}, X \to\infty,
\end{multline}
where $\Ord(\chi(\p))$ denotes the multiplicative order of the root of unity $\chi(\p)$. 
\end{prop}

\begin{proof}
Assume $\chi(D_{\alpha})\not\in \{0, 1\}$, so $\m$ is coprime with $D_{\alpha}$ and $D_{\alpha} \not\in T_{L/K}^{\m}$, hence not an ideal norm in $L$. By unique factorisation of ideals in $\OO_K$, we have 
\[
\chi(D_{\alpha})=\prod_{i=1}^n\chi_{\Art}(\p_{i})^{e_i} \neq 1.
\]

If we further have $\Ord(\chi(\p_i))$ dividing $\nu_{\p_i}(D_{\ell})$, then $\chi(\p_i)^{e_i}=1$. But since we assumed $\chi(D_{\alpha})\neq 1$, there must exist at least a prime ideal $\p_{i}$ with  $\Ord(\chi(\p_i)) \nmid \nu_{\p_i}(D_{\alpha})$ and $\p_i \mid D_{\alpha}$. 

We claim that it is sufficient to consider the set of $\OO$--principal prime ideals
\[
\Lambda=\{(\ell) \text{ prime in }\OO \text{ and }\Nm(\ell)\leq X \colon \ell \equiv \alpha \bmod \pi\OO, \alpha \in (\OO/\pi\OO)^{\times}\}.
\]
If $\ell \in \Lambda$, then by \Cref{Non PID periodicity} we have $D_{\ell} \sim D_{\alpha}$, so $\chi(D_{\ell})=\chi(D_{\alpha})\neq 0, 1$. From \cite[Proposition 7.20]{Cox}, there is a bijection between $\OO_K$--ideal prime to $f$ and $\OO$--ideal prime to $f$. There are only $O(1)$ $\OO$--prime ideals that are not coprime to $f$, so we may estimate the desired quantity from above with the prime ideal theorem (see \cite[Chapter 7.2, Corollary 1]{Narkiewicz}) and together with the equidistribution result, we obtain the asymptotic
\[
\frac{1}{\#(\OO/\pi\OO)^{\times}}\cdot\frac{1}{h(\OO)}\cdot\frac{X}{\log X}.
\]
In the case of $\Cl(K)=1$, the exact same argument applies to $B_{\alpha}$, except we consider $B_{\alpha} \bmod \m$ instead. 
\end{proof}
\section{Proof of \Cref{Main ideal norm result}}\label{Proof of sieve main result}
We now prove our main result. The method is very similar to the original one in \cite{Brauer}, except few changes in the justifications. Let $\chi$ be a group character as described in (\ref{character}) and assume it satisfies the hypotheses in \Cref{Lambda estimate}. In this section, we choose finite sets $S_1$ and $S_2$ such that $S_1$ consists of $\OO_K$--prime ideals dividing $\m$, and all primes $\p$ at which $E$ has bad reduction, and $S_2$ contains $\OO$--prime ideals dividing the conductor of the order $f=[\OO_F \colon \OO]$. Let 

\begin{align*}
\mathcal{P}
&= \{\text{primes } \p \not\in S_1 \colon \Ann(P \bmod \p) \text{ is a principal prime ideal in } \OO, \\
&\qquad \text{ and } \Ord(\chi(\p)) \nmid \nu_{\p}\!\left([\Ann(P \bmod \p)]P\right)\}, \\
\Lambda
&= \{\Ann(P \bmod \p)\colon \p \in \mathcal{P}\}, \\
T
&=\{\p \text{ prime in } \OO_{K}\colon \nu_{\p}(P)>0\}, \\
\ell_{\p}&=\Ann(P \bmod \p) .
\end{align*}

We first prove the following estimate, which will be used in counting the number of multiples of $P$ to be removed in the sieve process. 

\begin{lemma}\label{Lambda size}
Then 
\begin{multline}
\#\{\ell \in \Lambda \colon \ell \not\in S_2, \text{ a principal ideal in } \OO \text{ and }\Nm(\ell)\leq X\} \\
\geq \frac{1}{\#(\OO/\pi\OO)^{\times}}\cdot\frac{1}{h(\OO)}\cdot\frac{X}{\log X} - \#S_1 -\#S_2 - \#T,
\end{multline}
where $\pi$ is the period of the sequence $D_{\alpha}$ (resp.\ ($B_{\alpha}$) in the ray class group $\Cl_{K}^{\m}$ (resp.\ modulo $\m$). 
\end{lemma}

\begin{proof}
As in the original proof, we will give a lower bound to $\#\Lambda$ by estimating the number of prime elements in $\Lambda$ only (note that $\Ann(P \bmod \p)$ does not need to be a prime ideal). To do so, we consider the multiples $\left[\ell\right]P$ where $\ell$ runs over all principal prime ideals in $\OO$. If $\p$ is a prime ideal in $\OO_K$ such that $\nu_{\p}\left(\left[\ell\right]P\right) > \nu_{\p}(P) \geq 0$, then we have either $\ell=\Ann(P \bmod \p)$ or $P \equiv O \bmod \p$ (equivalently, $\p \in T$).

The former is obvious; if $P \equiv O \bmod \p$, then we claim that $\ell\OO_K = \p$, or equivalently $\p \cap \OO = \ell$. Indeed, if $\p \in T$, then we must have $P \in E_{\nu_{\p}(P)}(K)$, $P \not\in E_{\nu_{\p}(P)+1}(K)$ and $[\ell]P \in E_{\nu_{\p}(P)+1}(K)$. In other words, $\ell$ is an element in $\Ann(P)$ in the $\OO$--module $E_{\nu_{\p}(P)}(L)/E_{\nu_{\p}(P)+1}(L)$. Let $W=E_{\nu_{\p}(P)}(K)/E_{\nu_{\p}(P)+1}(K)$. From \cite[Corollary 2.10]{Streng}, we know that we have an $\OO$--module isomorphism $W \cong \OO_{K}/\p$. This implies that $W$ is a simple $\OO$--module, therefore $\Ann(P)$ is a proper maximal ideal in $\OO$ (by assumption, $\Ann(P)$ cannot be the entire endomorphism ring $\OO$). 

Overall, we now have a containment of prime ideals $\ell \subseteq \Ann(P \bmod \p) \subsetneq \OO$ in $\OO$. Since there are only $\#S_2=O(1)$ primes that are not coprime to $f$, the conductor of $\OO$, we may assume that both $\Ann(P \bmod \p)$ and $\ell$ are coprime to $f$. It follows that both $\Ann(P \bmod \p)$ and $\ell$ are maximal ideals of $\OO$. Therefore, as $\OO$--ideals, we must have $\Ann(P \bmod \p)=\ell$, which gives the isomorphism 
\begin{equation}\label{isomorphism module}
\OO/\ell \cong \OO/\Ann(P) = M \cong \OO_K/\p.
\end{equation}

Since $f\OO + \ell = \OO$, $\ell\OO_F$ is an $\OO_F$-ideal prime to $f$ of the same norm and hence we also have the isomorphism 
\[
\OO_F/\ell\OO_F \cong \OO_K/\p,
\]
which tells us that $\left[\OO_K/\p \colon \OO_F/\ell\OO_F\right]=1$. Let $p$ be the rational prime below $\ell\OO_F$ and $\p$, then we have two possible cases: if $p$ is inert or ramifies in $F$, then $p=\ell\OO_F$ or $\p=\left(\ell\OO_F\right)^2$; both implies that $\p=\ell\OO_F$ and hence $\p \cap \OO = \ell$. If $p$ splits completely in $F$, then $\p \cap \OO = \ell \text{ or } \bar{\ell}$. But if $\ell$ is a prime in $\OO$, then so as $\bar{\ell}$ and therefore the argument also applies to $\bar{\ell}$, with the corresponding $\OO_K$--prime being $\bar{\p}$. So without loss of generality, we showed that if $P \equiv O \bmod \p$, then $\p \cap \OO =\ell$. This means counting the size of $T$ is equivalent to counting the number of principal prime ideals $\ell$ in $\OO$ that fail to be the annilator of $P \bmod \p$ in the $\OO$--module $E_{0}(K)/E_{1}(K)$. 

From the proof of \Cref{Lambda estimate}, we know there exists a set of $\OO$-prime $\ell$ with $\Norm(\ell) \leq X$ whose cardinality can be bounded from below by the asymptotic
\[
\frac{1}{\#(\OO/\pi\OO)^{\times}}\cdot\frac{1}{h(\OO)}\frac{X}{\log X}
\]
such that $ \Ord\left(\chi(\p)\right) \nmid \nu_{\p}(D_{\ell})$ and $\p \nmid \m$ for some $\OO_K$--prime $\p$. For these $\OO$--primes $\ell$, we necessarily have $\nu_{\p}(D_{\ell})>0$. If $\nu_{\p}([\ell]P)=\nu_{\p}(P)$, then $\p \in T$ and $\ell = \p \cap \OO$. Overall, by excluding $\#S_1$, $\#S_2$ and $\#T$ of the $\OO$--primes $\ell$, these primes are now principal primes that lie in $\Lambda$ with $\p \not\in S_1$, hence give the result.
\end{proof}

The next lemma is a modification of the classical elliptic sieve, originating from \cite[Lemma 4.3]{Brauer}. It tells us which multiples of $P$ fail the norm condition outlined in \Cref{ray class group and norm}. It serves as the tool that connects the elliptic sieve and the $\p$--adic property of points on elliptic curves. 

\begin{lemma}
Let $\p$ be a prime in $\mathcal{P}$. For $\alpha \in \OO$, if $\alpha \in \ell_{\p}$ with $\ell_{\p}$ coprime to $f$ and further satisfying
\[
\left(\alpha\ell_{\p}^{-1}\right)\OO_{K}+\p=\OO_K,
\]
then $\Ord\left(\chi(\p)\right) \nmid \nu_{\p}\left([\alpha]P\right)$, where $\ell_{\p}^{-1}$ is a fractional ideal in $\OO$. 
\end{lemma}

\begin{proof}
First note that $\alpha\ell_{\p}^{-1}$ is really an integral ideal in $\OO$: since $\alpha \in \ell_{\p}$ and $\ell_{\p}$ does not lie above the conductor $f=[\OO_{F} \colon \OO]$, we have $\left(\alpha\OO\right)\ell_{\p}^{-1} \subseteq \ell_{\p}\ell_{\p}^{-1} = \OO$. 

We assumed that $\p$ is coprime with $\alpha\ell_{\p}^{-1}\OO_K$, so applying \Cref{valuation} to $[\ell_{\p}]P \in E_{1}(K)$ gives
\[
\nu_{\p}\left([\alpha]P\right)=\nu_{\p}\left(\alpha\ell_{\p}^{-1}\right)+\nu_{\p}\left([\ell_{\p}]P\right)=\nu_{\p}\left([\ell_{\p}]P\right),
\]
which is not divisible by $\Ord\left(\chi(\p)\right)$ by the definition of $\mathcal{P}$. 
\end{proof}

\begin{remark}
In \Cref{motivation sieve}, we mentioned that this version of the elliptic sieve differs from that of Kowalski's in \cite{Kowalski} in the $\p$--adic information required. Interested readers should compare the lemma above with that in \cite[Lemma 4.7]{Kowalski} to see that the former really provides information about modulo $\p^{n_{\p}}$, while the latter only needs to know modulo $\p$. 
\end{remark}

In other words, we now know that if $y\left([\alpha]P\right)$ is an ideal norm from $L$, $\alpha \in \OO$ has to belong to the set
\[
\{\alpha \in \OO \colon \alpha \not\in \ell_{\p} \text{ or } \alpha\OO_k \equiv 0 \bmod \p\left(\ell_{\p}\OO_k\right), \text{ where } \p \in \mathcal{P}\},
\]
whose cardinality can be bounded above by $N_{0}(B)+N_{1}(B)$, where 
\begin{align*}
N_{0}
&= \{\Norm_{F/\Q}(\alpha) \leq B \colon \alpha \in \OO, \alpha \not\equiv 0 \bmod \ell \text{ for all } \ell \in \Lambda \text{ with } \log(B) \leq \Norm_{F/\Q}(\ell) \leq B\}, \text{ and }\\
N_{1}
&=\#\{\Norm(\alpha\OO_{K})\leq B\colon \alpha \equiv 0 \bmod (\ell_{\p}\p)\OO_K \text{ for some } \p\in\mathcal{P} \text{ with } \log(B)\leq \Norm(\ell_{\p})\leq B\}.
\end{align*}

\begin{remark}
Note that due to the congruence conditions, we are treating $\alpha$ as an element in $\OO_F$ and $\OO_K$ in $N_0$ and $N_1$ respectively, which we are allowed to do so under the assumption $F \subseteq K$. 
\end{remark}

\begin{lemma}
We have 
\[
N_{1}(B) \ll \frac{1}{[K\colon F]}\frac{B}{\left(\log B\right)^{[K\colon F]}}.
\]
\end{lemma}

\begin{proof}
By definition, the set can be bounded from above by
\begin{align*}
N_{1}(B)
&\ll \sum_{\log B \leq \Norm(\ell_{\p}) \leq B} \#\{\Norm(\alpha) \leq B \colon \alpha \equiv 0 \bmod (\ell_{\p}\p)\OO_{K}\} \\
&\ll \sum_{\substack{\log B \leq \Norm(\ell) \leq B \\ \ell \in \Lambda}} \frac{B}{\Norm_{K/\Q}(\ell_\p\p)},
\end{align*}
where the bound follows from the fact that we have specified one residue class in $\left(\OO_K/\ell_{\p}\p\OO_K\right)^{\times}$. Note that here we have interpreted $\ell_{\p}$ as an element in $\OO_K$ rather than in $\OO$. 

Denote $\tilde{E}$ to be $E/K$ reduction modulo $\p$. Since $E/K$ has complex multiplication by $\OO$, there is an isomorphism 
\[
\End_{\C}(E) \xrightarrow{\sim} \End{\overline{\F_\p}}(\tilde{E}),
\]
where $\overline{F_\p}$ denotes the algebraic closure of the residue field $\OO_K/\p$, and there exists a prime $\eta \in \OO$ such that $\eta \xmapsto{\sim} \Frob_\p \in \End{\overline{F_\p}}(\tilde{E})$.

By \cite[Theorem V.1.1]{Silverman1},
\begin{align*}
\#\tilde{E}(\F_{\p})
&=\Deg\left(\mathbbm{1}-\Frob_{\p}\right) \\
&=\Deg(1-\eta) \\
&=\Norm(1-\eta).
\end{align*}

i.e. $[1-\eta]P \equiv O \bmod \p$. Since $\ell_\p = \Ann(P \bmod \p)$, we must have $1-\eta \in \ell_p\OO_K$, therefore $\Norm_{F/\Q}(1-\eta) \geq \Norm_{F/\Q}(\ell_\p)$. Using Hasse's bound, we obtain 
\[
\Norm_{F/\Q}(\ell_\p)\leq \Norm_{F/\Q}(1-\eta)=\#\tilde{E}(\overline{\F_{\p}})\leq \Norm_{K/\Q}(\p)+1+2\sqrt{\Norm_{K/\Q}(\p)} \ll \Norm_{K/\Q}(\p),
\]
which then iimplies $\frac{1}{\Norm_{K/\Q}(\p)} \ll \frac{1}{\Norm_{F/\Q}(\ell_\p)}$. So overall, this yields
\begin{align*}
N_{1}(B)
&\ll \sum_{\substack{\log B \leq \Norm(\ell) \leq B \\ \ell \in \Lambda}} \frac{B}{\Norm_{K/\Q}(\ell_\p)\Norm_{K/\Q}(\p)} \\
&\ll \sum_{\substack{\log B \leq \Norm(\ell) \leq B \\ \ell \in \Lambda}} \frac{B}{\Norm_{F/\Q}(\ell_\p)\Norm_{K/\Q}(\ell_{\p})} \\
&\ll B\sum_{\substack{\log B \leq \Norm(\ell) \leq B \\ \ell \in \Lambda}} \frac{1}{\Norm_{F/\Q}(\ell_\p)^{[K\colon F]+1}} \\
&\ll B\sum_{n \geq \log B} \frac{1}{n^{[K\colon F]+1}} \\
&\ll \frac{1}{[K\colon F]}\frac{B}{\left(\log B\right)^{[K\colon F]}},
\end{align*}
where the final estimate follows from extending the sum over all elements with norm at least $\log B$ and the use of integral bound.
\qedhere
\end{proof}

\begin{lemma}
An upper bound for $N_{0}(B)$ is 
\[
 \prod_{\substack{\log B \leq \Norm_{F/\Q}(\q) \leq B \\ \ell \in \Lambda, \q \cap \OO=\ell}}
\left(1-\frac{1}{\Norm_{F/\Q}(\q)}\right)
\ll \frac{B\log|\Delta_{F}| \log\log B}{\left(\log B\right)^{\left(1/(\OO/\pi\OO)^{\times}\right)\left(1/h(\OO)\right)}}\exp(\#S_1+\#S_2+\#T),
\]
\end{lemma}

\begin{proof}
The proof proceeds in the exact same way as the original proof in \cite[Lemma 4.5]{Brauer}. Since we have assumed that $\ell$ does not divide the conductor $f=[\OO_{F}\colon \OO]$, we always have $\ell\OO_{F}=\q$ for some prime ideal $\q$ in $\OO_F$ of the same norm (see \cite[Proposition 7.20]{Cox}). Therefore, we can apply Selberg sieve for number fields to the primes in $\OO_F$. Using the version in \cite[Theorem 14, Lemmata 15, 16]{Debaene}, this gives us 
\begin{align*}
N_{0}(B) 
& \leq \frac{2\pi B}{\sqrt{\Delta_{F}}}\prod_{\substack{\log B \leq \Norm_{F/\Q}(\q) \leq \sqrt{B} \\ \ell \in \Lambda, \q \cap \OO=\ell}} \left(1-\frac{1}{\Norm_{F/\Q}(\q)}\right) + O\left(B(\log B)^6\right) \\
&\ll B \prod_{\substack{\log B \leq \Norm_{F/\Q}(\q) \leq \sqrt{B} \\ \ell \in \Lambda, \q \cap \OO=\ell}}
\left(1-\frac{1}{\Norm_{F/\Q}(\q)}\right).
\end{align*}

By Mertens' theorem for number fields (see \cite[Theorem 1 (C1)]{Garcia}), we have
\[
\prod_{\substack{\Norm_{F/\Q}(\q) \leq \log B \\ \ell \in \Lambda, \q \cap \OO=\ell}}
\left(1-\frac{1}{\Norm_{F/\Q}(\q)}\right)^{-1}
\ll 
\prod_{\substack{\Norm_{F/\Q}(\q) \leq \log B}}
\left(1-\frac{1}{\Norm_{F/\Q}(\q)}\right)^{-1}
\ll \log|\Delta_{F}| \log\log B.
\]

Hence, it remains to bound
\[
\prod_{\substack{\Norm_{F/\Q}(\q) \leq \sqrt{B} \\ \ell \in \Lambda, \q \cap \OO=\ell}}
\left(1-\frac{1}{\Norm_{F/\Q}(\q)}\right)
\]
from above. To do so, notice that
\[
\begin{aligned}
\log \prod_{\substack{\Norm_{F/\Q}(\q) \leq \sqrt{B} \\ \ell \in \Lambda, \q \cap \OO=\ell}}
\left(1-\frac{1}{\Norm_{F/\Q}(\q)}\right)^{-1}
&= -\sum_{\substack{\Norm_{F/\Q}(\ell) \leq \sqrt{B} \\ \ell \in \Lambda, \q \cap \OO=\ell}}
\log\left(1-\frac{1}{\Norm_{F/\Q}(\ell)}\right) \\
&\geq \sum_{\substack{\Norm_{F/\Q}(\q) \leq \sqrt{B} \\ \ell \in \Lambda, \q \cap \OO=\ell}}
\frac{1}{\Norm_{F/\Q}(\q)} 
\end{aligned}
\]

Provided that $\Norm_{F/\Q}(\q)=p$, or equivalently $p$ splits in $F$, then the sum will diverge and giving us a non-trivial bound. By partial summation, this is 
\[
 \sum_{\substack{\Norm_{F/\Q}(\q) \leq \sqrt{B} \\ \q \cap \OO = \ell,\ \ell \in \Lambda}}
\frac{1}{\sqrt{B}}+\int_{1}^{\sqrt{B}} \sum_{\substack{\Norm_{F/\Q}(\q) \leq u \\ \q \cap \OO = \ell,\ \ell \in \Lambda}}
\frac{1}{u^2} du.
\]

By Chebotarev Density Theorem, the elements of $\Lambda$ that correspond to rational primes that split completely in $F$ has density $1/2$, so \Cref{Lambda size} shows that 
\[
\sum_{\substack{\Norm_{F/\Q}(\q) \leq \sqrt{B} \\ \q \cap \OO = \ell,\ \ell \in \Lambda}}
\frac{1}{\sqrt{B}}=\frac{1}{\#(\OO/\pi\OO)^{\times}}\cdot\frac{1}{h(\OO)}\frac{2}{\log B}-\frac{\#S_1+\#S_2+\#T}{\sqrt{B}}, \text{ and }
\]
\begin{align*}
\int_{1}^{\sqrt{B}} \sum_{\substack{\Norm_{F/\Q}(\q) \leq u \\ \q \cap \OO = \ell,\ \ell \in \Lambda}}
\frac{1}{u^2} du
&=\frac{1}{\#(\OO/\pi\OO)^{\times}}\cdot\frac{1}{h(\OO)}\int_{1}^{\sqrt{B}} \frac{1}{u\log u} du - (\#S_1+\#S_2+\#T)\int_{1}^{\sqrt{B}} \frac{1}{u^2}du \\
&=\frac{1}{\#(\OO/\pi\OO)^{\times}}\cdot\frac{1}{h(\OO)}\log\log(\sqrt{B})+(\#S_1+\#S_2+\#T)\left(1-\frac{1}{\sqrt{B}}\right) \\
&=\frac{1}{\#(\OO/\pi\OO)^{\times}}\cdot\frac{1}{h(\OO)}\left(\log\log(B)+O(1)\right)-(\#S_1+\#S_2+\#T)\left(1-\frac{1}{\sqrt{B}}\right).
\end{align*}

Overall, we have 
\[
\log \prod_{\substack{\Norm_{F/\Q}(\q) \leq \sqrt{B} \\ \ell \in \Lambda, \q \cap \OO=\ell}}
\left(1-\frac{1}{\Norm_{F/\Q}(\q)}\right)^{-1}
\geq \frac{1}{\#(\OO/\pi\OO)^{\times}}\cdot\frac{1}{h(\OO)}\left(\log\log(B)+O(1)\right)-(\#S_1+\#S_2+\#T).
\]

Exponentiating and take reciprocals give 
\[
 \prod_{\substack{\log B \leq \Norm_{F/\Q}(\q) \leq B \\ \ell \in \Lambda, \q \cap \OO=\ell}}
\left(1-\frac{1}{\Norm_{F/\Q}(\q)}\right)
\ll \frac{B\log|\Delta_{F}| \log\log B}{\left(\log B\right)^{\left(1/\#(\OO/\pi\OO)^{\times}\right)\left(1/h(\OO)\right)}}\exp(\#S_1+\#S_2+\#T).
\]
and hence the result.
\end{proof}

\begin{proof}[Proof of \Cref{Main ideal norm result}]
Let $L/K$ be a finite abelian extension. If $y([\alpha]P) \in K$ is an ideal norm in $L$, then $y([\alpha]P) \in T_{L/K}^{\m}$, where the conductor of the extension $\f_{L/K}$ divides the modulus $\m$. Let $\chi$ be a group character as described in (\ref{character}). As stated in the hypotheses of the theorem, we assume the existence of $\alpha_0 \in \OO$ with $\alpha_0 \in (\OO/\pi\OO)^{\times}$ such that $\chi(D_{\alpha_{0}})\not\in\{0, 1\}$. Hence, we may apply \Cref{Lambda estimate} the previous lemmata in this section, giving us the estimate
\begin{align*}
&\#\{\alpha \in \OO\colon |\alpha|\leq B, y(\alpha P) \text{ is an ideal norm from } L\} \\
&\ll N_0(B)+N_1(B)\\
&\ll \frac{1}{[K\colon F]}\frac{B}{\left(\log B\right)^{[K\colon F]}}+\frac{B\log|\Delta_{F}| \log\log B}{\left(\log B\right)^{\left(1/\#(\OO/\pi\OO)^{\times}\right)\left(1/h(\OO)\right)}}\exp(\#S_1+\#S_2+\#T).
\end{align*}

Note that $\# S_1 = O_L(1)$, $\# S_2 = O_E(1)$, and $\log|\Delta_F| = O_E(1)$ (since the choice of curve determines the endomorphism ring, and hence the imaginary quadratic field $F$). In addition, it is obvious that
\[
\frac{1}{\#(\OO/\pi\OO)^{\times}}\frac{1}{h(\OO)} \leq 1 \leq [K\colon F],
\]
therefore the first term gets dominated by the second term. Our estimate hence becomes
\[
\#\{\alpha \in \OO\colon |\alpha|\leq B, y(\alpha P) \text{ is a norm from } L\} 
\ll_{E, L}\frac{B \log\log B}{\left(\log B\right)^{\left(1/\#(\OO/\pi\OO)^{\times}\right)\left(1/h(\OO)\right)}}\exp(\#T).
\]

To get the final bound in the theorem, which depends explicitly on $P$, we proceed to show that
\begin{equation}\label{size of T}
\#T\ll_{E} \frac{\hat{h}(P)}{\log\hat{h}(P)},
\end{equation}
whose proof is almost identical to that of \cite[Lemma 4.6]{Brauer}, except that now we use the ideal version for the estimates.
If $\nu_{\p}(P)>0$, then $\p$ is a factor of $D_1$. Therefore, we have 
\[
\#T=\omega\left(\prod_{\nu_{\p}(P)>0}\p\right)=\omega(D_{1}),
\]
where $\omega(\A)$ denotes the number of distinct prime ideal factors of the ideal $\A$. 

By definition, we have $\Norm(D_{1})\leq h(x(P))$, where $h(\cdot)$ is the naive height function. Theorem 1.1 of \cite{Silverman5} implies that for all $P \in E(K)$, 
\[
\left|\hat{h}(P)-\frac{1}{2}h(x(P))\right|<2h(E),
\]
therefore $\log\Norm(D_{1})\leq\hat{h}(P)+O_{E}(1)$. $\omega(D_{1})$ can be bounded by 
\[
\omega(D_{1}) \ll \frac{\log\Norm(D_{1})}{\log\log\Norm(D_{1})},
\]
and since $\hat{h}(P)\gg_{E} 1$ by \cite[Theorem VIII.9.10(a)]{Silverman1}, this completes the proof of (\ref{size of T}). 

Finally, recall that $\pi(\m)$ is an upper bound for the period of the sequence $(D_{\alpha})_{\alpha \in \OO}$ (resp.\ $(B_{\alpha})_{\alpha \in \OO}$) in the group $I_{K}^{\m}/T_{L/K}^{\m}$ (resp.\ modulo $\m$). Let $\p_1^{k_1}\ldots\p_n^{k_n}$ be the unique ideal factorisation of $\m$, then from \Cref{PID periodicity} and \Cref{Non PID periodicity} we have 
\[
\pi(\m)=\begin{cases}
			\phi(\m)\prod_{\p \mid \m}M \Ord \left([M]P \bmod \p^{k+\nu_{\p}([M]P)}\right), & \text{if $\Cl(K)=1$}\\
            \bigcap\limits_{i=1}^{n}\left(h^{\p_i^{k_i}}(K)\right)M \Ord \left([M]P \bmod \p_i^{k_i+\nu_{\p}([M]P)}\right), & \text{otherwise},
		 \end{cases}
\]
where $M$ is an element in the ideal $\mathfrak{M}(P)$ as described in \Cref{CM EDS}. Note that we have the bound
\[
 \Ord \left([M]P \bmod \p^{k+\nu_{\p}([M]P)}\right) \leq |E_{k+\nu_{\p}([M]P)+1}(K_{\p})/E_{k+\nu_{\p}([M]P)}(K_{\p})|=\Norm(\p),
\]
and by definition the size of $M$ can be bounded by $\prod_{\p}|E(K_{\p}/E_{0}(K_{\p})|$, uniformly with respect to $P$. Therefore, in both cases we can bound the size of $\pi$ by $O_{E, L}(1)$. 

It is also obvious that $h(\OO)=O_{E}(1)$. Hence, our final estimate is 
\[
\#\{\alpha \in \OO\colon |\alpha|\leq B, y(\alpha P) \text{ is an ideal norm from } L\} 
\ll_{E, P, L}\frac{B}{\left(\log B\right)^{\omega-\varepsilon}}
\]
for any $\varepsilon >0$, where $\omega=\omega(E, L)>0$, and the implicit constant is $\exp\left(O_{E, L}\left(\hat{h}(P)/\log\hat{h}(P)\right)\right)$, completing the proof. 
\end{proof}
\section{Worked examples}\label{sieve example}
Recall that our main theorem contains a technical condition concerning the existence of a group character and a denominator ideal satisfying $\chi(D_{\alpha})\notin \{0, 1\}$. This assumption is necessary for the conclusion to hold. In this section, we illustrate our main theorem with three examples. The first two consider the elliptic curve $E \colon y^2 = x^3 - 2x$, which has complex multiplication by $\mathbb{Z}[i]$, defined over number fields with different class numbers. The third example demonstrates that if the hypotheses of \Cref{Main ideal norm result} are not satisfied, then the conclusion could fail. In this section, we abbreviate elliptic denominator sequences as EDS for simplicity. 

\begin{example}\label{example Q(i)}
In this example, we consider the elliptic curve 
\[
E/\Q(i)\colon y^2 = x^3 - 2x,
\]
This curve has complex multiplication by $\OO=\Z[i]$ and the Mordell--Weil group $\Z^2 \oplus \Z/2\Z$, with $E(\Q(i))_{\text{tors}}=\{O, (0, 0)\}$ and generators $P_1=(-1,1)$ and $P_2=(i-1,-2)$ (see the corresponding \href{https://www.lmfdb.org/EllipticCurve/2.0.4.1/4096.1/CMd/1}{LMFDB entry} for more details). The EDS associated to $P_1$, starting at $B_0$, is listed in the second column of \Cref{Table EDS}.

We now apply the main theorem to $L/\Q(i)$, where $L$ is the ray class field of $\Q(i)$ with modulus $\m=(1+i)^4=\p^4=(4)$, and therefore $T_{L/K}^{\m}=P_K^{\m}$. A SageMath computation yields $\#\Cl_{\Q(i)}^{\m}=[L\colon \Q(i)]=2$. Since $L$ is a quadratic extension of $\Q(i)$ with $1+i$ being the only ramified prime, LMFDB allows us to identify $L$ as $\Q(\zeta_8)$, the $8$--th cyclotomic field (see \href{https://www.lmfdb.org/NumberField/?degree=4&ram_quantifier=exactly&ram_primes=2&ur_primes=3&subfield=x\%5E2\%2B1&field_is=ab&cm_field=yes}{LMFDB search result}). 

Furthermore, as $\#\Cl_{\Q(i)}^{\m}=2$, we deduce that
\[
\Cl_{\Q(i)}^{\m}=\{[1], [1+2i]\}\cong \left(\OO/\m\OO\right)^{\times}/\{\pm1, \pm i\}.
\]
Next, to find a character $\chi\colon \Cl_{\Q(i)}^{\m}\rightarrow\C^{\times}$ and an $\alpha\in\OO]$ satisfying $(\alpha, \pi(\m))=1$ and $\chi(B_{\alpha})\notin\{0,1\}$, we consider the character table of $\Cl_{\Q(i)}^{\m}$:
\[
\begin{array}{c|cc}
\Cl_{\Q(i)}^{\m} 	& [1] 		& [1+2i] \\ \hline
\chi_0 		& 1 		& 1 \\
\chi_1 		& 1 		& -1
\end{array}
\]
Clearly, only the non--trivial character $\chi_1$ could possibly satisfy the hypothesis of our main theorem. With reference to \Cref{Table EDS} again, we see that $\chi_{1}(B_{2+i})=\chi_1([2-i])=-1$. Note that $P_1$ has everywhere good reduction, so we may take $M=1$ and $k=4$ in \Cref{PID periodicity}, this gives us an upper bound for the period of the EDS: 
\[
\pi(\m)=\phi(\p^4)r(\p^4)=\Nm(\p^3)(\Nm(\p)-1)M \Ord \left([M]P \bmod \p^{2+\nu_{\p}([M]P)}\right)=8(4)=2^5,
\]
where $\Ord \left([M]P \bmod \p^{2+\nu_{\p}([M]P)}\right)=4$ as observed from \Cref{Table EDS}. Therefore, $\alpha=2+i$ is a valid choice. Since all hypotheses are satisfied, \Cref{Main ideal norm result} shows that 
\[
\#\left\{
\alpha\in \Z[i] \colon |\alpha|\leq B,\;
y(\alpha P_1)\text{ is an ideal norm from } L
\right\}
\ll_{\varepsilon}
\frac{B}{(\log B)^{\frac{1}{2048}-\varepsilon}}.
\]
\end{example}

\begin{example} \label{sieve example 2}
As previously mentioned, in this example we consider the elliptic curve 
\[
E/\Q(\theta)\colon y^2 = x^3 - 2x,
\]
where $\theta=(1+i)\sqrt{-\frac{5}{2}}$ has minimal polynomial $\theta^4+25$ (one may also verify that $\Q(\theta)=\Q(i,\sqrt{-10})$). SageMath computations show that $E/\Q(\theta)$ also has rank $2$. Since $\Q(i)\subset\Q(\theta)$, $P_1 \in E(\Q(\theta))$ and we may choose $P_1=(-1, 1)$ and $P_2=(i-1,-2)$ as the generators of the free part, as well as considering the EDS generated by $P_1$ again as in \Cref{example Q(i)}. However, note that $\Q(\theta)$ has class number $2$, as verified by SageMath, so we need to consider the periodicity of ray classes corresponding to the denominator ideals $D_{\alpha}$ instead. The EDS associated to $P_1$, starting at $D_0$, is listed in \Cref{Table EDS}.

Similar to \Cref{example Q(i)}, we wish to apply the main theorem to $L/\Q(\theta)$, where $L$ is the ray class field for the modulus $\m=\theta+i-1$, a prime lying above $3$. SageMath computations show that $\#\Cl_{\Q(\theta)}^{\m}=[L\colon \Q(\theta)]=4$, which by \Cref{Non PID periodicity}, also bounds the period of the ray classes corresponding to $D_{\alpha}$. To find a suitable ray class group character, we need to compute the ray class group explicitly. To do so, note that the primes $2\pm i \in \Z[i]$ ramify in $\Q(\theta)$:
\[
(5)=(2+i)(2-i)=\left(5, \sqrt{10}+2i-1\right)^2\left(5, 2(\sqrt{-10}+i)+1\right)^2=\q_1^2\q_2^2. 
\]
Since $3\in \m$, we may deduce that $\q_1^2\equiv 2+i  \equiv i-1 \not\equiv 1 \bmod \m$ and conclude that $\Cl_{\Q(\theta)}^{\m}\cong C_4$, which has the character table below. 
\[
\begin{array}{c|cccc}
\Cl_{\Q(\theta)}^{\m}	& [1] 	& [\q_1] 	& [\q_1] ^2 	& [\q_1] ^3 \\ \hline
\chi_0 				& 1 		& 1 		& 1 			& 1 \\
\chi_1 				& 1 		& i 		& -1 			& -i \\
\chi_2 				& 1 		& -1 		& 1 			& -1 \\
\chi_3 				& 1 		& -i 		& -1 			& i
\end{array}
\]
It remains to choose a suitable denominator ideal $D_{\alpha}$ and character $\chi$ such that $(\alpha, \pi(\m))=1$ and $\chi(D_{\alpha})\notin\{0,1\}$. Here, we use \Cref{Non PID periodicity} and recall that $P$ has good reduction everywhere, so we may take $M$=1 and yield 
\[
\pi(\m)=h^{\m}(\Q(\theta))r(\m)=4\Ord(P \bmod \m)=16.
\]
Hence, we may choose $\alpha=2+i$. However, in this case, only $\chi_1$ and $\chi_3$ satisfy the hypothesis of the main theorem, as $\chi_0(D_{2+i})=\chi_2(D_{2+i})=1$. Overall, all hypotheses are satisfied so \Cref{Main ideal norm result} shows that
\[
\#\left\{
\alpha\in \Z[i] \colon |\alpha|\leq B,\;
y(\alpha P_1)\text{ is an ideal norm from } L
\right\}
\ll_{\varepsilon}
\frac{B}{(\log B)^{\frac{1}{8}-\varepsilon}}.
\]
\end{example}

\renewcommand{\arraystretch}{1.5}
\begin{xltabular}{\linewidth}{c|c|>{\centering\arraybackslash}X}
\hline
$\alpha \in \mathbb{Z}[i]$
& $B_{\alpha}$ ($E$ over $\Q(i)$)
& $D_{\alpha}$ ($E$ over $\Q(\theta)$) \\
\hline
\endfirsthead

\hline
$\alpha \in \mathbb{Z}[i]$
& $B_{\alpha}$ ($E$ over $\Q(i)$)
& $D_{\alpha}$ ($E$ over $\Q(\theta)$) \\
\hline
\endhead

\hline
\multicolumn{3}{r}{\textit{Continued on next page}} \\
\endfoot

\hline
\caption{EDS associated with $E \colon y^2 = x^3 - 2x$ and the point $P=(-1, 1)$, indexed by $\mathcal{O} = \mathbb{Z}[i]$, where $E$ is defined over $\mathbb{Q}(i)$ and $\Q(\theta)$ in columns $2$ and $3$ respectively. Note that column $2$ shows a sequence of elements, while column $3$ shows a sequence of ideals. All entries are either taken from \cite[Table 1]{Streng} or computed by SageMath.}
\label{Table EDS}
\endlastfoot

$0$ 
& $0$ 
& $0$ \\	

$1$ 
& $1$ 
& $1$ \\	

$i$ 
& $1$ 
& $1$ \\	

$1+i$
& $1+i$
& $\left(2, \frac{\sqrt{-10}}{2}(i+1) +1\right)^2$ \\

$1-i$ 
& $1+i$
& $\left(2, \frac{\sqrt{-10}}{2}(i+1) +1\right)^2$ \\

$1+2i$ 
& $2+i$
& $\left(5, \sqrt{10}+2i-1\right)^2$ \\

$1-2i$ 
& $2-i$
& $\left(5, 2(\sqrt{-10}+i)+1\right)^2$ \\        

$2$ 
& $2$
& $\left(2, \frac{\sqrt{-10}}{2}(i+1) +1\right)^4$ \\

$2i$ 
& $2$
& $\left(2, \frac{\sqrt{-10}}{2}(i+1) +1\right)^4$ \\    	

$2+i$ 
& $2-i$
& $\left(5, 2(\sqrt{-10}+i)+1\right)^2$ \\

$2-i$ 
& $2+i$
& $\left(5, \sqrt{10}+2i-1\right)^2$  \\

$2+2i$ 
& $3(1+i)^{3}$
& $\left(2, \frac{\sqrt{-10}}{2}(i+1) +1\right)^6
   \left(\frac{\sqrt{-10}}{2}(i-1)+i+1\right)$
   \newline
   $\left(\frac{\sqrt{-10}}{2}(i+1)+i-1\right)$ \\

$2-2i$ 
& $3(1+i)^{3}$
& $\left(2, \frac{\sqrt{-10}}{2}(i+1) +1\right)^6
   \left(\frac{\sqrt{-10}}{2}(i-1)+i+1\right)$
   \newline
   $\left(\frac{\sqrt{-10}}{2}(i+1)+i-1\right)$ \\

$3$ 
& $(3+2i)(3-2i)$
& $\left(13, \frac{\sqrt{-10}}{2}(i+1)+1\right)
   \left(13, \frac{\sqrt{-10}}{2}(i+1)-1\right)$
   \newline
   $\left(13, \frac{\sqrt{-10}}{2}(i+1)+5\right)
   \left(13, \frac{\sqrt{-10}}{2}(i+1)-5\right)$ \\

$3+i$ 
& $(1+i)(2+i)(4-i)$
& $(4-i)\left(2,\frac{\sqrt{-10}}{2}(i+1)+1\right)^2$
   \newline
   $\left(5,\sqrt{10}+2i-1\right)^2$ \\

$3-i$ 
& $(1+i)(2+i)(4-i)$
& $(4-i)\left(2,\frac{\sqrt{-10}}{2}(i+1)+1\right)^2$
   \newline
   $\left(5,\sqrt{10}+2i-1\right)^2$ \\

$3+2i$
& $(5+4i)(6-i)$
& $\left(\frac{\sqrt{-10}}{2}(i+1)+2i\right)
   \left(\frac{\sqrt{-10}}{2}(i+1)-2i\right)$
   \newline
   $\left(37,\frac{\sqrt{-10}}{2}(i+1)+9\right)
   \left(37,\frac{\sqrt{-10}}{2}(i+1)-9\right)$ \\

$3+3i$
& $(1+i)(3+2i)(3-2i)$
& $\left(2,\frac{\sqrt{-10}}{2}(i+1)+1\right)^2
   \left(13,\frac{\sqrt{-10}}{2}(i+1)+1\right)$
   \newline
   $\left(13,\frac{\sqrt{-10}}{2}(i+1)-1\right)
   \left(13,\frac{\sqrt{-10}}{2}(i+1)+5\right)$
   \newline
   $\left(13,\frac{\sqrt{-10}}{2}(i+1)-5\right)$ \\

$4$
& $(1+i)^4(3)(7)$
& $\left(2,\frac{\sqrt{-10}}{2}(1+i)+1\right)^8
   \left(\frac{\sqrt{-10}}{2}(i-1)+i+1\right)$
   \newline
   $\left(\frac{\sqrt{-10}}{2}(i+1)+i-1\right)
   \left(\frac{\sqrt{-10}}{2}(i+1)+i+1\right)$
   \newline
   $\left(\frac{\sqrt{-10}}{2}(i+1)-i-1\right)$ \\

$4+i$
& $(5-2i)(14-i)$
& $(5-2i)
   \left(197,\frac{\sqrt{-10}}{2}(1+i)+18\right)$
   \newline
   $\left(197,\frac{\sqrt{-10}}{2}(1+i)-18\right)$ \\

$4+2i$
& $(i+1)^2(4+i)(2-i)(16+9i)$
& $\left(2,\frac{\sqrt{-10}}{2}(1+i)+1\right)^4
   \left(5,2\sqrt{-10}+i-2\right)^2$
\newline
$(4-i)(-9i-16)$ \\

$4+3i$
& $(2+i)(14-9i)(32+23i)$
& $\left(5,-i\sqrt{-10}-2i+1\right)^2
   \left(277,\frac{\sqrt{-10}}{2}(1+i)+85\right)$
\newline
$\left(277,\frac{\sqrt{-10}}{2}(1+i)-85\right)
   (-32i+23)$ \\

$4+4i$
& $(1+i)^5(3)(7)(8+7i)(8-7i)$
& $\left(2,\frac{\sqrt{-10}}{2}(1+i)+1\right)^{10}
   \left(\frac{\sqrt{-10}}{2}(i-1)+i+1\right)$
\newline
$\left(\frac{\sqrt{-10}}{2}(1+i)+i-1\right)
   \left(\frac{\sqrt{-10}}{2}(1+i)+i+1\right)$
\newline
$\left(\frac{\sqrt{-10}}{2}(1+i)-i-1\right)
   (8i-7)(-7i+8)$ \\
$\vdots$
& $\vdots$
& $\vdots$
\end{xltabular}

\begin{example}\label{sieve example 3}
In this example, we consider the elliptic curve 
\[
E/\Q(\sqrt{-2})\colon y^2=x^3+x^2-3x+1.
\]
This curve has complex multiplication by $\OO=\Z[\sqrt{-2}]$ and and the Mordell--Weil group $\Z^2 \oplus \Z/2\Z$. The discriminant is $\Delta=2^9=-(\sqrt{-2})^{18}$, so the only prime of bad reduction is $\p=\sqrt{-2}$ (see the corresponding \href{https://www.lmfdb.org/EllipticCurve/2.0.8.1/1024.1/CMa/1}{LMFDB entry} for more details). We now consider the EDS associated to the point $P=(0, -1)$ with 
\[
[\sqrt{-2}]P=\left(\frac{1}{2}, \frac{1}{4\sqrt{-2}}\right), \text{ where } [\sqrt{-2}](x, y)=\left( \frac{-x^2 - 1}{2(x - 1)}, \frac{\sqrt{-2}y(x^2 - 2x - 1)}{4(x - 1)^2} \right), 
\]
 where we have listed some of the values of the EDS in \Cref{Table EDS sqrt-2} (see \cite[Chapter II, Proposition 2.3.1]{Silverman2} for the derivation of the formula for $[\sqrt{-2}]$) . Note that the points $P$ and $[\sqrt{-2}]P$ are $\Z$--linearly independent, so they may be taken as generators of the free part of the Mordell-Weil group. 

We would like to apply the main theorem to this choice of $E$ and $P$ with the extension $L/\Q(\sqrt{-2})$, where $L$ is the ray class field of $\Q(\sqrt{-2})$ for the modulus $\m=\p^2$. A SageMath computation yields $\#\Cl_{\Q(\sqrt{-2})}^{\m}=[L\colon \Q(i)]=2$. Since $L$ is a quadratic extension of $\Q(\sqrt{-2})$ with $\sqrt{-2}$ being the only ramified prime, LMFDB allows us to identify $L$ as $\Q(\zeta_8)$, the $8$--th cyclotomic field (see \href{https://www.lmfdb.org/NumberField/?degree=4&ram_quantifier=exactly&ram_primes=2&subfield=x\%5E2\%2B2&field_is=ab}{LMFDB search result}). 

Since $\Cl_{\Q(\sqrt{-2})}^{\m} \cong C_2$, so it has the same character table as in \Cref{example Q(i)}. However, we claim that the hypothesis of the main theorem is not satisfied here. To see this, note that $P$ and $[\sqrt{-2}]P$ are both nonsingular modulo $\mathfrak{p}$, so we set $M=1$ and $k=2$ in \Cref{PID periodicity}. Refer to \Cref{Table EDS sqrt-2}, we see that $[2]P = \left(\frac{5}{4}, -\frac{7}{8}\right)$, so we obtain 
\[
\pi(\m)=\phi(\p^2)r(\p^2)=\Nm(\p)(\Nm(\p)-1)M \Ord \left([M]P \bmod \p^{2+\nu_{\p}([M]P)}\right)=4.
\] 
This means that the EDS associated to $P$, $B_{\alpha}$, is periodic modulo $\m$ with period bounded from above by $4$ : for any $\alpha, a+b\sqrt{-2} \in \OO$,
\[
B_{\alpha + (a + b\sqrt{-2})} \equiv B_{\alpha} \bmod \m \text{ whenever } 4 \mid a, b, 
\]
Therefore, once we know all the character values evaluated at $B_{a+b\sqrt{-2}}$ for $0\leq a, b \leq 4$, then the claim can be verified. For the unique nontrivial character, the corresponding values are listed in \Cref{table EDS fail sieve} below, where we have used SageMath to compute the coordinates of $[\alpha]P$ and deduce the entries in \Cref{table EDS fail sieve} (see the relevant code in \Cref{appendix sieve}). 
\begin{table}[H]
\centering
\begin{tabular}{c|ccccc}
4 & 0 & 1 & 0 & 1 & 0 \\
3 & 0 & 1 & 0 & 1 & 0 \\
2 & 0 & 1 & 0 & 1 & 0 \\
1 & 0 & 1 & 0 & 1 & 0 \\
0 & 0 & 1 & 0 & 1 & 0 \\
\hline
$b \backslash a$ & 0 & 1 & 2 & 3 & 4
\end{tabular}
\caption{Values of $\chi(B_{\alpha})$ for $\alpha = a + b\sqrt{-2} \in \mathcal{O}=\Z[\sqrt{-2}]$, where $\chi$ is the unique nontrivial ray class character of $\mathbb{Q}(\sqrt{-2})$ for the modulus $\mathfrak{p}^2 = (2)$ and $B_{\alpha}$ is the EDS associated to $E/\Q(\sqrt{-2})\colon y^2=x^3+x^2-3x+1$ and the point $P=(0, -1)$. The horizontal and vertical axes correspond to $a, b \in \{0, \ldots, 4\}$, respectively.}
\label{table EDS fail sieve}
\end{table}
Observing \Cref{table EDS fail sieve}, one can notice that whenever $2$ divides $a$, the corresponding column consists of zeros only. This should not be a surprise as it follows from the $\OO$--module filtration structure (see \Cref{Streng module}) immediately: $[\sqrt{-2}]P \in E_1\left(\Q(\sqrt{-2})_{\nu}\right)$, yet $P \in E_0\left(\Q(\sqrt{-2})_{\nu}\right)\setminus E_1\left(\Q(\sqrt{-2})_{\nu}\right)$, where $\nu$ is the valuation corresponding to the prime $\p = \sqrt{-2}$. Overall, using the periodicity argument above and \Cref{table EDS fail sieve}, we conclude that for any $\alpha=a+b\sqrt{-2} \in \OO$,
$$
\chi(B_{\alpha})=\begin{cases}
			0, & \text{if $2\mid a \iff \p \mid \alpha$},\\
            1, & \text{otherwise.}
		 \end{cases}
$$
This demonstrates that when the hypothesis of \Cref{Main ideal norm result} is not satisfied, the conclusion could fail. In particular, since $E(\Q(\sqrt{-2})$ is of rank $2$, generated by $P$ and $[\sqrt{-2}]P$, we actually showed that approximately $50\%$ of the points in $E(\Q(\sqrt{-2})$ have $y$--coordinate which are elements in $T_{L/K}^{\m}$. Since our method is unable to determine if an element in $T_{L/K}^{\m}$ is an ideal norm from $L=\Q(\zeta_8)$ or not, our conclusion does not hold. 

\renewcommand{\arraystretch}{1.5}

\begin{xltabular}{\linewidth}{c|>{\centering\arraybackslash}X}

\hline
$\alpha \in \mathbb{Z}[\sqrt{-2}]$ & $B_{\alpha}$ \\
\hline
\endfirsthead

\hline
$\alpha \in \mathbb{Z}[\sqrt{-2}]$ & $B_{\alpha}$ \\
\hline
\endhead

\hline
\multicolumn{2}{r}{\textit{Continued on next page}} \\
\endfoot

\hline
\caption{EDS associated with $E /\Q(\sqrt{-2})\colon y^2 = x^3+x^2-3x+1$ and the point $P=(0, -1)$, indexed by $\alpha \in \OO=\mathbb{Z}[\sqrt{-2}]$, which is a sequence of elements. All entries are computed using Sagemath.}
\label{Table EDS sqrt-2}
\endlastfoot

$0$ 
& $0$\\

$\sqrt{-2}$ 
& $(\sqrt{-2})$\\

$2\sqrt{-2}$ 
& $(\sqrt{-2})^3$\\

$3\sqrt{-2}$ 
& $(\sqrt{-2})(-\sqrt{-2}-1)^2(\sqrt{-2}-1)^2$\\

$4\sqrt{-2}$ 
& $(\sqrt{-2})^5(7)(31)$\\

$1$ 
& $1$\\

$1+\sqrt{-2}$ 
& $1$\\

$1+2\sqrt{-2}$ 
& $(-4\sqrt{-2}-3)$\\

$1+3\sqrt{-2}$ 
& $(16\sqrt{-2}+57)$\\

$1+4\sqrt{-2}$ 
& $(-8\sqrt{-2}-3)(96\sqrt{-2}+91)$\\

$2$ 
& $(\sqrt{-2})^2$\\

$2+\sqrt{-2}$ 
& $(\sqrt{-2})(-\sqrt{-2}-1)$\\

$2+2\sqrt{-2}$ 
& $(\sqrt{-2})^2(\sqrt{-2}-1)(-2\sqrt{-2}+3)$\\

$2+3\sqrt{-2}$ 
& $(\sqrt{-2})(3\sqrt{-2}+1)(-8\sqrt{-2}+3)$\\

$2+4\sqrt{-2}$ 
& $(\sqrt{-2})^2(-\sqrt{-2}-1)(\sqrt{-2}-3)(2\sqrt{-2}+3)(-4\sqrt{-2}-3)(6\sqrt{-2}+13)$\\

$3$ 
& $5$ \\

$3+\sqrt{-2}$ 
& $-8\sqrt{-2}-3$\\

$3+2\sqrt{-2}$ 
& $-12\sqrt{-2}+17$\\

$3+3\sqrt{-2}$ 
& $5(4\sqrt{-2}-3)(-12\sqrt{-2}+5)$\\

$3+4\sqrt{-2}$ 
& $-408\sqrt{-2}-2603$\\

$4$ 
& $(\sqrt{-2})^4(7)$\\

$4+\sqrt{-2}$ 
& $(\sqrt{-2})(\sqrt{-2}-1)(\sqrt{-2}+3)(4\sqrt{-2}-3)$\\

$4+2\sqrt{-2}$ 
& $(\sqrt{-2})^3(-\sqrt{-2}-1)(2\sqrt{-2}+3)(4\sqrt{-2}+9)$\\

$4+3\sqrt{-2}$ 
& $(\sqrt{-2})(-3\sqrt{-2}-5)(\sqrt{-2}-9)(12\sqrt{-2}+17)$\\

$4+4\sqrt{-2}$ 
& $(\sqrt{-2})^4(\sqrt{-2}-1)(7)(-2\sqrt{-2}+3)(-4\sqrt{-2}+9)(8\sqrt{-2}+33)$\\

$\vdots$ 
& $\vdots$

\end{xltabular}
\end{example}

\subsubsection*{Acknowledgements}
I would like to express my gratitude to my supervisors, Dr.\ Simon L.\ Rydin Myerson, and Dr.\ Helena Verill for their incredible support over the past year, along with their insightful comments, and for carefully reading earlier versions of this paper.
\appendix 
\section{SageMath Code}\label{appendix sieve}
For reference, this appendix contains the SageMath code used for the computations in \Cref{sieve example}.

\subsection{Ray class number computation}
The following code computes the ray class numbers needed for \Cref{sieve example}.
\begin{lstlisting}[language=Python]
K.<z> = QuadraticField(-1)

L.<w> = K.extension(x^2 + 10)
# construct an extension over K

M.<r> = QuadraticField(-1)
y = polygen(M)

p = K.primes_above(2)[0]
# Takes the first prime in the list

q = L.primes_above(3)[0]
# Takes the first prime in the list

r = M.primes_above(2)[0]
# Takes the first prime in the list

print(f'Ray class number for Example 5.1 is {(p^4).ray_class_number()}')
print(f'Ray class number for Example 5.2 is {q.ray_class_number()}')
print(f'Ray class number for Example 5.3 is {r.ray_class_number()}')
\end{lstlisting}

The output is shown as below.
\begin{lstlisting}
Ray class number for Example 5.1 is 2
Ray class number for Example 5.2 is 4
Ray class number for Example 5.3 is 2
\end{lstlisting}

\subsection{Denominator ideal computation}
The following code computes the coordinates of the multiples of points considered in \Cref{Table EDS sqrt-2}.
\begin{lstlisting}[language=Python]
K.<r> = QuadraticField(-2)
E = EllipticCurve(K, [0, 1, 0, -3, 1])

P = E(0, -1)
x1_coord = P[0]
y1_coord = P[1]

# Define the point rP on the elliptic curve
s = 1 - 2*(x1_coord - 1)^(-2)

#Here we define the [sqrt{-2}] map
x2_coord = r^(-2) * (x1_coord - 1 + 4 + 2/(x1_coord - 1)) + 1
y2_coord = y1_coord * r^(-3) * s

# [r] denotes the [sqrt{-2}] map
rP = E([x2_coord, y2_coord])

for m in range(1, 9):        # m = 1 to 8
    for n in range(1, 9):    # n = 1 to 8
        R = m * P + n * rP
        print(f"[{m}+{n}r]P = {R}")
        if R[0]:
            print("x-coordinate factorisation:", K.factor(R[0]))
        else:
            print("x-coordinate factorisation: 0")
        print("==============")
\end{lstlisting}
We omit the output here due to the length of the list. Part of the results is given in \Cref{Table EDS sqrt-2}.
\bibliographystyle{abbrv}
\bibliography{Reference2}

\end{document}